\documentclass{article} 
\usepackage{latexsym}
\usepackage{amsmath, rotating}
\usepackage{natbib}
\usepackage{amsmath, rotating, color}
\usepackage{amsfonts,amssymb, amsthm}
\usepackage{psfrag}
\usepackage{pictex}
\newcommand\unnumberedfootnote[1]{ %
        \let\temp=\thefootnote %
        \renewcommand{\thefootnote}{}%
        \footnote{#1}%
        \let\thefootnote=\temp%
        \addtocounter{footnote}{-1}}

\newtheorem{proposition}{Proposition}[section]

\theoremstyle{definition}
\newtheorem{remark}[proposition]{Remark}

\numberwithin{equation}{section}
\begin{document}
\title{\LARGE How often does the ratchet click?\\ Facts, heuristics, asymptotics.}

\thispagestyle{empty}

\author{{\sc by A. Etheridge\thanks{Research supported in part 
by EPSRC GR/T19537/01.}, P. Pfaffelhuber\thanks{Travel support
      from DFG, Bilateral Research Group
      FOR 498.} and A. Wakolbinger} \\[2ex]
  \emph{University of Oxford, Ludwig-Maximilians University Munich} \\
  \emph{and Goethe-University Frankfurt}\vspace*{-5ex}}  \date{}


\maketitle

\unnumberedfootnote{\emph{AMS 2000 subject classification.} {\tt
    92D15, 60J70} (Primary) {\tt 60K35, 60H35, 60H10} (Secondary).}

\unnumberedfootnote{\emph{Keywords and phrases.} Muller's ratchet,
  selection, mutation, Fleming-Viot process, Wright-Fisher model,
  diffusion approximation} \centerline {\date{17. July 2007}}
\begin{abstract}
\noindent 
The evolutionary force of recombination is lacking in asexually
reproducing populations. As a consequence, the population can suffer an 
irreversible accumulation of deleterious mutations, 
a phenomenon known as \emph{Muller's
  ratchet}. We formulate
discrete and continuous time versions of Muller's ratchet. Inspired by Haigh's (1978)
analysis of a dynamical system which arises in the limit of large
populations, we identify the parameter
$\gamma = N\lambda/(Ns\cdot \log(N\lambda))$ as most important for the
speed of accumulation of deleterious mutations. Here $N$ is population
size, $s$ is the selection coefficient and $\lambda $ is the
deleterious mutation rate. For large parts of the
parameter range, measuring time in units of size $N$,
deleterious mutations accumulate according to a power
law in $N\lambda$ with exponent $\gamma$ if $\gamma\ge 0.5$. For
$\gamma<0.5$ mutations cannot accumulate. We obtain
diffusion approximations for three different parameter regimes,
depending on the speed of the ratchet. Our approximations shed
new light on analyses of Stephan et al.~(1993) and Gordo \&
Charlesworth~(2000). The heuristics leading to the approximations
are supported by simulations.
\end{abstract}

\section{Introduction}
\label{intro}
Muller's ratchet is a mechanism that has been suggested as an explanation 
for the evolution of sex (Maynard Smith~1978).  
\nocite{MaynardSmith1978}
The idea is simple; in an asexually reproducing 
population chromosomes are passed down as indivisible blocks and so the 
number of deleterious mutations accumulated along any ancestral line in the
population can only increase.  When everyone in the current `best' 
class has accumulated at least one additional bad mutation then the 
minimum mutational load in the population increases: the ratchet clicks.
In a sexually reproducing population this is no longer the case; because
of recombination between parental genomes a parent
carrying a high mutational load 
can have offspring with {\em fewer} deleterious mutations.  The
high cost of sexual reproduction is thus offset by the benefits of 
inhibiting the ratchet.  
Equally the ratchet provides a possible explanation for the 
degeneration of $Y$ chromosomes in sexual organisms (e.g. Charlesworth~1978, 
1996; Rice~1994; 
Charlesworth \& Charlesworth~1997, 1998).
\nocite{Charlesworth1978}
\nocite{CharlesworthCharlesworth1997}
\nocite{CharlesworthCharlesworth1998}
\nocite{Charlesworth1996}
\nocite{Rice1994}
However, in order to assess its real 
biological importance one should establish
under what circumstances Muller's
ratchet will have an evolutionary effect.  In particular, 
{\em how many generations will it take for an asexually reproducing
population to lose its current best class?}
In other words, what is the rate of the ratchet?

In spite of the substantial literature devoted to the ratchet (see
Loewe~2006 for an extensive bibliography), even in the simplest mathematical
\nocite{Loewe2006}
models a closed form expression
for the rate remains elusive.  
Instead various approximations have
been proposed which fit well with simulations for particular parameter regimes.
The analysis presented here unifies these approximations into a single 
framework and provides a more detailed mathematical understanding of their
regions of validity.

The simplest mathematical model of the ratchet was formulated in the pioneering
work of Haigh~(1978).  Consider an asexual population of constant size $N$.  
\nocite{Haigh1978}
The population evolves according to classical Wright-Fisher dynamics.  
Thus each of the $N$
individuals in the $(t+1)$st generation independently
chooses a parent 
from the individuals in the $t$th generation.  The probability that an
individual 
which has accumulated $k$ mutations is selected is proportional to its
{\em relative} fitness, $(1-s)^k$.  
The number of mutations carried by the offspring is then $k+J$ where $J$ is 
an (independent) Poisson random variable with mean $\lambda$.

Haigh identifies $n_0=Ne^{-\lambda/s}$ as an approximation (at large times)
for the
total number of individuals carrying the least number of mutations 
and finds numerical evidence of a linear relationship between $n_0$
and the average time between clicks of the ratchet,
at least for `intermediate'
values of $n_0$ (which he quantifies as $n_0>1$ and less than $25$,
say).  On the other hand, for increasing values of $n_0$ Stephan et
al.~(1993) note the increasing \nocite{StephanChaoSmale1993}
importance of $s$ for the rate of the ratchet. The simulations of
Gordo \& Charlesworth~(2000) \nocite{GordoCharlesworth2000a} also
suggest that for $n_0$ fixed at a large value the ratchet can run at
very different rates.  They focus on parameter ranges that may be the
most relevant to the problem of the degeneration of large
non-recombining portions of chromosomes.

In our approach we use diffusion approximations to identify another
parameter as being an important factor in determining the rate of the
ratchet. We define
\begin{align}\label{scaling}
  \gamma:=\frac{N\lambda}{Ns\log(N\lambda)}.
\end{align}
Notice that $n_0=N(N\lambda)^{-\gamma}$.  In these parameters one
can reinterpret Haigh's empirical results as saying that if we measure time in
units of size $N$, then the rate of the ratchet follows a power law in
$N\lambda$.  In fact our main observation in this note is that
for a substantial portion of parameter space
(which we shall quantify a little more precisely later) we have the
following
\\\\
\noindent {\bf Rule of Thumb.\,}{\em The rate of the ratchet is of the
  order $N^{\gamma -1}\lambda^\gamma$ for $\gamma \in (1/2, 1)$,
  whereas it is exponentially slow in $(N\lambda)^{1-\gamma}$ for
  $\gamma <1/2$.}



~ 

There are two novelties here.  First, the abrupt change in behaviour
at $\gamma=1/2$ and second the power law interpretation of the rate
for $\gamma\in (1/2,1)$. As an appetiser, Figure \ref{fig1}
illustrates that this behaviour really is reflected in simulated data;
see also \S\ref{simulations}.

The rule of thumb breaks down for two scenarios: first, if $\gamma>1$
then for large $N\lambda$ we have $n_0<1$ and so our arguments, which
are based on diffusion approximations for the size of the best class, will
break down.
This parameter regime, which leads to very frequent clicks of the ratchet,
was studied by Gessler (1995). \nocite{Gessler1995} Second, 
if $N\lambda$ is {\em too} large then we see 
the transition from exponentially rare clicks to frequent clicks takes
place at larger values of $\gamma$.


\begin{figure}
\begin{center}
(A)
\vspace{-4ex}

\includegraphics[width=4in]{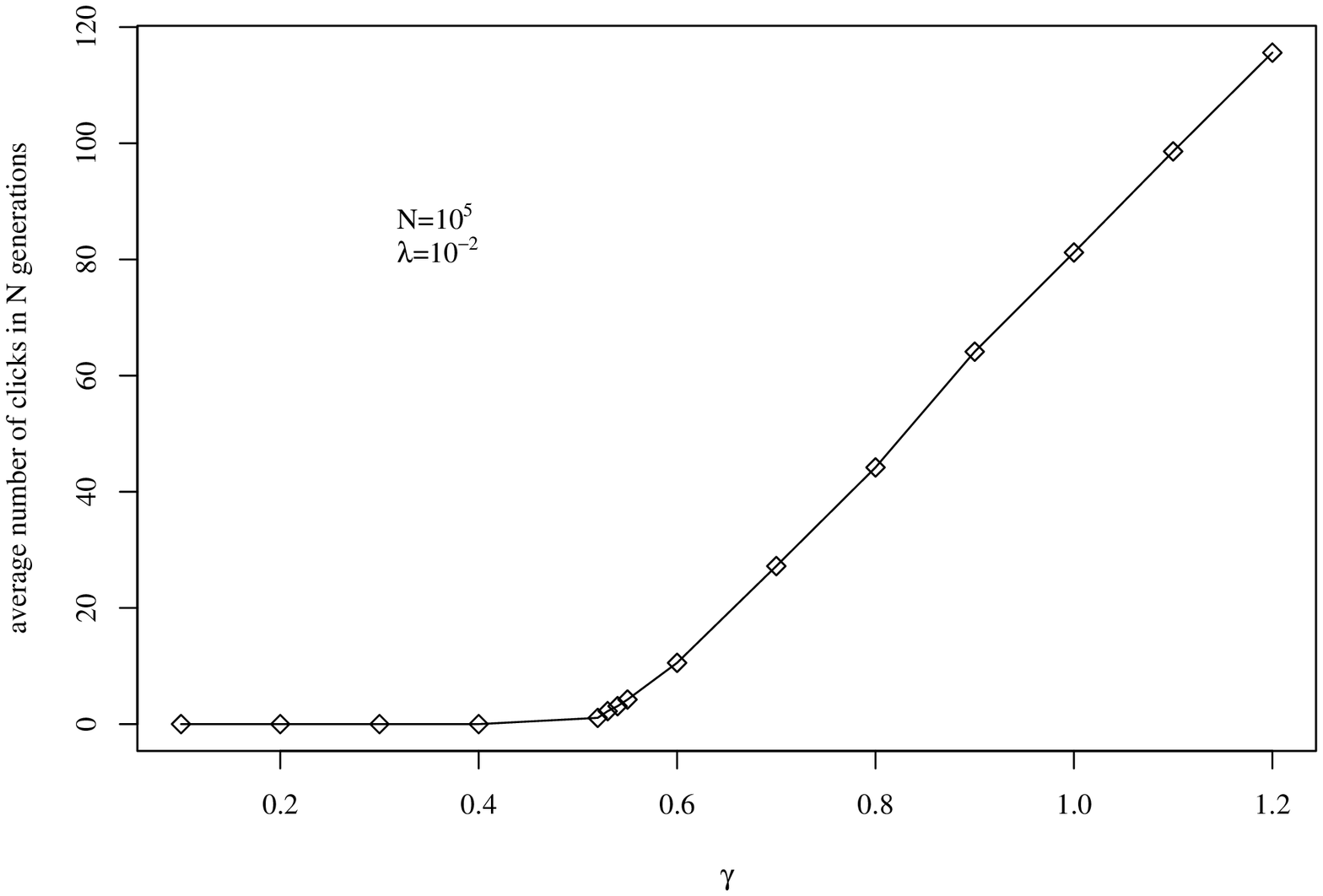}

(B)
\vspace{-4ex}

\includegraphics[width=4in]{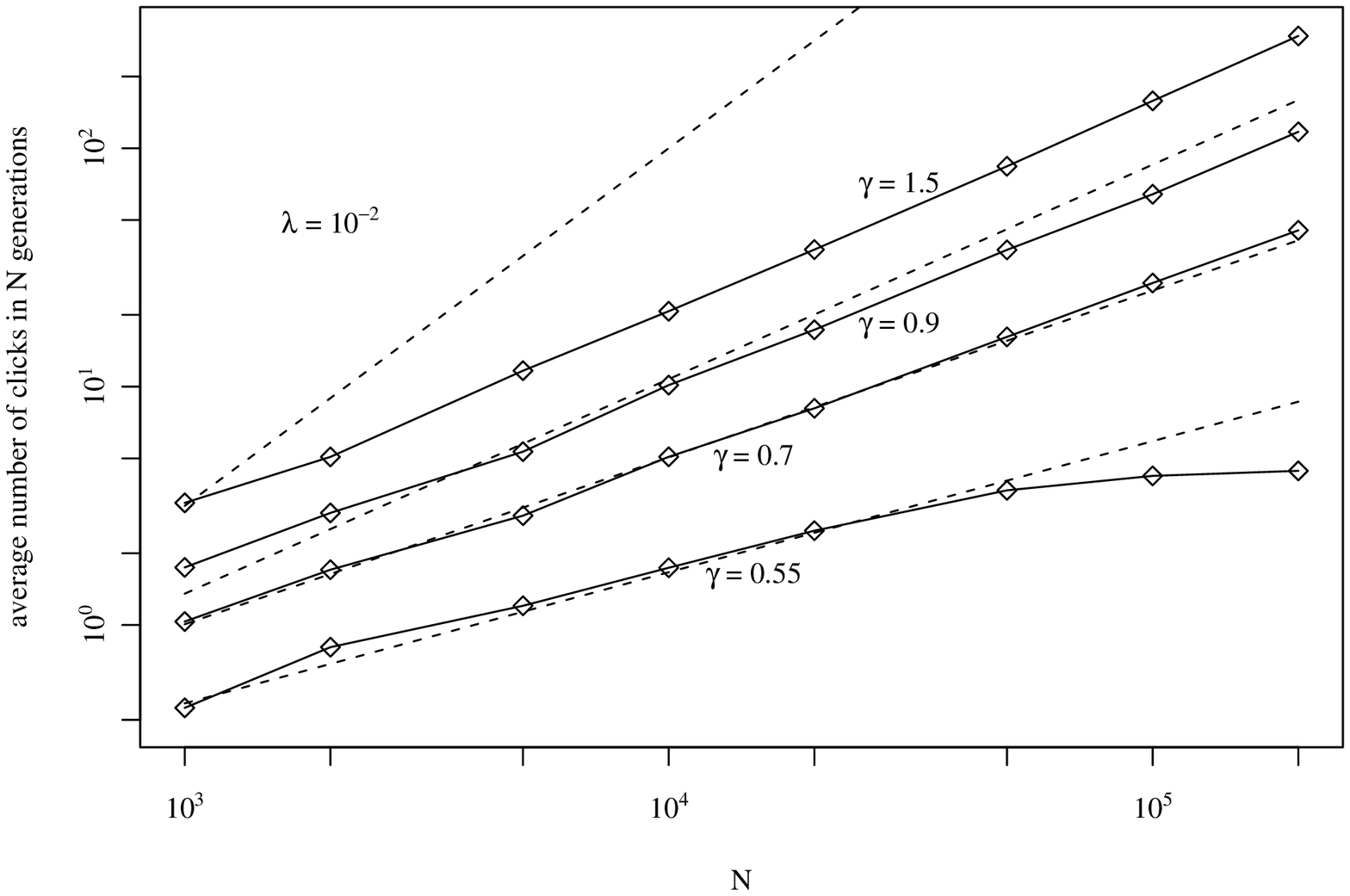}
\end{center}
\caption{\label{fig1}(A) We plot the rate of the ratchet against
  $\gamma$, where time is measured in units of $N$ generations. As
  predicted by our rule of thumb we see a sharp change of behaviour
  around $\gamma=0.5$.  (B) We see the power-law behaviour of the rate
  for various values of $\gamma$.  The solid lines are given by
  simulation of a Wright-Fisher model.  The dashed lines fit the
  prediction that the time between clicks is a constant times
  $N(N\lambda)^{-\gamma}$ for $\gamma >0.5$.  Note that this breaks
  down for $\gamma >1$.}
\end{figure}

The rest of this note is laid out as follows.  In \S\ref{haigh} we
review the work of Haigh~(1978).  Whereas Haigh's work focuses on
discrete time dynamics, in \S\ref{Fleming-Viot} we write down instead
a (continuous time) Fleming-Viot diffusion approximation to the model
whose behaviour captures the dynamics of large populations when $s$
and $\lambda$ are small.  We then pass in \S\ref{deterministic
  dynamics} to the infinite population (deterministic) limit.  This
system can be solved exactly and by, in the spirit of Haigh, using
this deterministic system to estimate the behaviour of the `bulk' of
the population we obtain in \S\ref{one-dimensional diffusion} a one
dimensional diffusion which approximates the size of the best class in
our Fleming-Viot system.  The drift in this one-dimensional diffusion
will take one of three forms depending upon whether the ratchet is
clicking rarely, at a moderate rate or frequently per $N$ generations
(but always rarely per generation).  Performing a scaling of this
diffusion allows us to predict the relationship between the parameters
$N\lambda$ and $Ns$ of the biological model and the value of $\gamma$
at which we can expect to see the phase transition from rare clicking
to power law behaviour of the rate of the ratchet.  In
\S\ref{simulations} we compare our predictions to simulated data, and
in \S\ref{discussion} we discuss the connection between our findings
and previous work of Stephan et al. (1993), Stephan \& Kim (2002) and
Gordo \& Charlesworth (2000).

\section{The discrete ratchet -- Haigh's approach} \label{haigh}

The population dynamics described in the introduction can be 
reformulated mathematically as follows.
Let $N$ be a fixed natural number (the population size), 
$\lambda > 0$ (the mutation parameter) 
and $s \in (0,1)$ (the selection coefficient).
The population is described by a stochastic process taking values in
$\mathcal P(\mathbb N_0)$, the simplex of probability weights 
on $\mathbb N_0$.
Suppose that 
${\bf x}(t) = (x_k(t))_{k=0,1,...}  \in \mathcal P(\mathbb N_0)$ is the 
vector of {\em type frequencies} (or {\em frequency profile}) in 
the $t$th generation (so for
example $Nx_k(t)$ individuals in the population carry exactly $k$
mutations). 
Let $H$ be an $\mathbb N_0$-valued random variable with $\mathbf P[H=k]$ 
proportional to $(1-s)^kx_k(t)$, let $J$ be a Poisson($\lambda$)-random 
variable 
independent of $H$ and let $K_1,..., K_N$ be independent copies of $H+J$. 
Then the random type frequencies in the next generation are
\begin{align}\label {DR} X_k(t+1) =  \frac1N \#\{i: K_i =k\}. 
\end{align}
We shall refer to this as the {\em ratchet dynamics in discrete time}.

First consider what happens as $N\rightarrow\infty$.  By the law of 
large numbers, 
\eqref{DR} results in the deterministic dynamics
\begin{align}\label{DDR} {\bf x}(t+1) 
 := 
\mathbf E_{\mathbf x(t)}[{\bf X}(t+1)].
\end{align} 
An important property of the dynamics
\eqref{DDR} is that vectors of Poisson weights are mapped to
vectors of Poisson weights. To see this, note that when $N\to\infty$
the right hand side of \eqref{DDR} is just the {\em law} of the random
variable $H+J$.  If $\mathbf x(t) =$ Poisson($\alpha$), the Poisson
distribution with mean $\alpha$, then $H$ is Poisson
($\alpha(1-s)$)-distributed and consequently $\mathbf x(t+1)$ is the law
of a Poisson($\alpha(1-s)+\lambda $) random variable.  We shall see in
\S\ref{deterministic dynamics} that the same is true of the continuous
time analogue of \eqref{DDR} and indeed we show there that
for every initial condition with $x_0>0$ the solution to the
continuous time equation converges to the stationary point
$$\pi := \mbox{Poisson}(\theta)$$ as $t\to\infty$, where
\begin{align*}\theta := \frac \lambda s.
\end{align*} 

Haigh's analysis of the finite population model focusses on the number of 
individuals in the best class.  
Let us write $k^*$ for the number
of mutations carried by individuals in this class.  In a finite population,
$k^*$ will increase with time, but the profile of frequencies 
{\em relative to $k^*$}, $\{X_{k+k^*}\}$, forms a recurrent Markov chain.  
We set
$$\mathbf Y := (Y_k)_{k=0,1,..} := (X_{k^*+k})_{k=0,1,..} $$  
and observe that since fitness is always
measured relative to the mean fitness in the population, between clicks of 
the ratchet the equation for the dynamics of 
${\mathbf Y}$ is precisely the same as that for ${\mathbf X}$ when $X_0>0$.
Suppose that after $t$ generations
there are $n_0(t) = N y_0(t)$ individuals in the best class.  
Then the probability of sampling a parent from this class
{\em and} not acquiring any additional mutations is
$$p_0(t):=(y_0(t)/W(t)) \, e^{-\lambda},$$ 
where 
\begin{align}\label{eq:fit}
  W(t) = \sum_{i=0}^\infty y_i(t)(1-s)^i.
\end{align}
Thus, given $y_0(t)$, the size of the best class in the next
generation has a binomial distribution with $N$ trials and success
probability $p_0(t)$, and so the evolution of the best class is
determined by $W(t)$, the mean fitness of the population.  We shall
see this property of $Y_0$ reflected in the diffusion approximation of
\S\ref{Fleming-Viot}.

A principal assumption of Haigh's analysis is that 
immediately before a click, the individuals of the current best class have 
all been distributed upon the other classes, in proportion to their 
Poisson weights.  Thus immediately after a click he takes the type 
frequencies (relative to the new best class) to be 
\begin{align}\label{pitilde}
\tilde \pi:= \frac 1{1-\pi_0} (\pi_1,\pi_2,\ldots)=  
\frac1{1-e^{-\theta}}\left( \theta e^{-\theta}, 
\frac{\theta^2}2 e^{-\theta}, ...\right).
\end{align}
The time until the next click is then subdivided into two phases.  During 
the first phase the deterministic dynamical system decays exponentially 
fast towards its Poisson equilibrium, swamping the randomness arising from 
the finite population size.  At the time that he proposes for the end of the 
first phase the size of the best class is approximately $1.6\pi_0$. 
The mean fitness of
the population has decreased by an amount which can also 
be readily estimated and,
combining this with a Poisson approximation to the binomial distribution, Haigh
proposes that (at least initially) during the second (longer)
phase the size of the best class should
be approximated by a Galton-Watson
branching process with a Poisson offspring distribution.

Haigh's original proposal was that since the mean fitness of the population 
(and consequently the 
mean number of offspring in the Galton Watson process)
changes only slowly during the second phase it 
should be taken to be constant throughout that phase.
Later refinements have modified Haigh's approach in two key ways.  First they
have worked with a diffusion approximation so that the Galton-Watson
process is replaced by a Feller diffusion and second, instead of taking 
a constant drift,
they look for a good approximation of the mean
fitness {\em given the size of the best class}, resulting in a 
Feller diffusion with {\em logistic} growth.  Our aim in the rest of this
paper is to unify these approximations in a single 
mathematical framework and discuss them in the light of
simulations.  
A crucial building block will be the
following extension of \eqref{pitilde}, which we call the {\em
  Poisson profile approximation} (or PPA) of $\mathbf Y$ based on
$Y_0$:
\begin{align}\label{PPA} \Pi(Y_0):= \Big(Y_0, \frac
  {1-Y_0}{1-\pi_0}(\pi_1,\pi_2,\ldots )\Big ).
\end{align}

As a first step, we now turn to a diffusion approximation for the 
full ratchet  dynamics \eqref{DR}.



\section{The Fleming-Viot diffusion}
\label{Fleming-Viot}

For large $N$ and small $\lambda$ and $s$, the following stochastic dynamics on 
$\mathcal P(\mathbb N_0)$ in continuous time captures the conditional 
expectation 
and variance of the discrete dynamics \eqref{DR}:
\begin{align}\label{FV}
dX_k = \left(\sum_{j}s(j-k)X_jX_k +\lambda (X_{k-1}-X_k)\right)dt + \sum_{j\neq k}\sqrt{\frac 1NX_jX_k} \, dW_{jk},\\ \nonumber  k=0,1,2,\ldots, 
\end{align}
where $X_{-1}:= 0$, and $(W_{jk})_{j>k}$ is an array of independent
standard Wiener processes with $W_{kj}:= -W_{jk}$.  This is, of
course, just the infinite-dimensional version of the standard
multi-dimensional Wright-Fisher diffusion. Existence of a process
solving \eqref{FV} can be established using a diffusion limit of the
discrete dynamics of \S\ref{haigh}. The coefficient $s(j-k)$ is the
fitness difference between type $k$ and type $j$,
$\lambda(X_{k-1}-X_k)$ is the flow into and out of class $k$ due to
mutation and the diffusion coefficients $\frac 1N X_jX_k$ reflect the
covariances due to multinomial sampling.

\begin{remark}
Often when one passes to a Fleming-Viot diffusion approximation, one 
measures time in units of size $N$ and correspondingly the 
parameters $s$ and $\lambda$ appear as $Ns$ and $N\lambda$.
Here we have {\em not} rescaled time, hence the factor of $\frac{1}{N}$ in
the noise and the unscaled parameters $s$ and $\lambda$ in the equations.
\hfill$\Box$
\end{remark}
\noindent
Writing
\begin{align*}
 M_1({\bf X}) := \sum_j jX_j
 \end{align*}
 for the first moment of $\bf X$, \eqref{FV} translates into
\begin{align}\label{eq:FV}
dX_k =\Big(s(M_1({\bf X})-k)-\lambda )X_k +\lambda X_{k-1}\Big) dt + \sum_{j\neq k}\sqrt{\frac 1NX_jX_k} \, dW_{jk}
\end{align}
In exactly the same way as in our discrete stochastic system, writing $k^*$ for
the number of mutations carried by individuals in the fittest class, one would 
like to think of the population as a travelling wave with profile
$Y_{k}=X_{k+k^*}$.
Notice in particular that 
\begin{equation*}
dY_0=(sM_1(\mathbf Y)-\lambda)Y_0dt+\sqrt{\frac{1}{N}Y_0(1-Y_0)}dW_0,
\end{equation*}
where $W_0$ is a standard Wiener process.  Thus, just as in Haigh's setting, 
the frequency of the best class is determined by the mean fitness of 
the population.

Substituting into \eqref{FV} one can obtain a stochastic equation for $M_1$,
$$dM_1=(\lambda -sM_2)dt+dG$$
where $M_2=\sum_j(j-M_1)^2X_j$ and the martingale $G$ has quadratic variation
$$d\langle G\rangle=\frac{1}{N}M_2dt.$$
Thus the speed of the wave is determined by the variance of the profile.
Similarly,
$$dM_2=(-\tfrac 1N M_2+(\lambda -sM_3))dt+dH$$
where $M_3=\sum_j(j-M_1)^3X_j$ and the martingale $H$ has quadratic
variation $d\langle H\rangle=\frac{1}{N}(M_4-M_2)dt$ with $M_4$
denoting the fourth centred moment and so on.

These equations for the centred moments 
are entirely analogous to those obtained 
by Higgs \& Woodcock~(1995) except that in the Fleming-Viot setting they are exact.  As
\nocite{HiggsWoodcock1995}
pointed out there, B\"urger~(1991) obtained similar equations to study the 
\nocite{Burger1991}
evolution of polygenic traits.  The difficulty in using these equations to study the rate
of the ratchet is of course that they are not closed: the equation for $M_k$
involves $M_{k+1}$ and so on.  Moreover, there is no obvious approximating
closed system.  By contrast, the infinite population limit, in which the noise is absent, turns out to have a closed solution.

\section{The infinite population limit}
\label{deterministic dynamics}

The continuous time analogue of \eqref{DDR} is the deterministic dynamical system
\begin{align}\label{cds}
dx_k =\Big((s(M_1({\bf x})-k)-\lambda )x_k +\lambda x_{k-1}\Big) dt, \quad k= 0,1,2,\dots
\end{align}
where $x_{-1}=0$, obtained by letting $N\to\infty$ in our Fleming-Viot
diffusion \eqref{eq:FV}. Our goal in this section is to solve this
system of equations. Note that Maia et al. (2003) have obtained a
complete solution of the corresponding discrete system following
\eqref{DDR}.  \nocite{MaiaBotelhoFontanari2003}

As we shall see in Proposition~\ref{prop1}, the stationary points of
the system are exactly the same as for \eqref{DDR}, that is $\bf x=
\pi$ and all its right shifts $ (\pi_{k-k^*})_{k=0,1,\ldots}$, $ k^*=
0,1,2,\dots $.  Since the Poisson distribution can be characterised as
the only distribution on $\mathbb N_0$ with all cumulants equal, it is
natural to transform \eqref{cds} into a system of equations for the
cumulants, $\kappa_k, k=1,2,\ldots$, of the vector $\bf x$.  The
cumulants are defined by the relation
\begin{align}\label{eq:siDef}
  \log \sum_{k=0}^\infty x_k e^{-\xi k} = \sum_{k=1}^\infty \kappa_k
  \frac{(-\xi)^k}{k!}.
\end{align} 
We assume $x_0>0$ and set
\begin{align}\label{eq:SI0}
\kappa_0 := -\log x_0.
\end{align}

\begin{proposition}\label{prop1}
  For $\kappa_k, k=0,1,2,\ldots$ as in
  \eqref{eq:siDef} and \eqref{eq:SI0} the system \eqref{cds} is
  equivalent to
  \begin{align*}
    \dot\kappa_k = -s\kappa_{k+1} + \lambda,\qquad\qquad k=0,1,2,\ldots
  \end{align*}
  Setting $\underline \kappa :=(\kappa_0,\kappa_1,\ldots)$ this system
  is solved by
\begin{equation}
\label{solution to cds}
   \underline \kappa = B \underline \kappa(0)^\top + \frac{\lambda}{s}(1-e^{-st})\underline 1,
  \quad B= (b_{ij})_{i,j=0,1,\ldots},\quad b_{ij} = \begin{cases}
    \frac{(-st)^{j-i}}{(j-i)!} & j\geq i\\0 & \text{otherwise.}
  \end{cases}
\end{equation}
In particular,
  \begin{align}\label{solutionx0}
    x_0(t) = e^{-\kappa_0(t)} = x_0(0) \frac{\exp\big( -
      \frac{\lambda}{s}(1-e^{-st})\big)}{\Big(\sum_{k=0}^\infty
      x_k(0) e^{-stk}\Big)}
  \end{align}
  and
  \begin{align}
    \label{solutionM}
    \kappa_1(t) = \sum_{k=0}^\infty k x_k(t) = \left.-\frac{\partial}{\partial\xi}
      \log \sum_{k=0}^\infty x_k(0) e^{-\xi k}\right|_{\xi=st}+
    \frac{\lambda}{s}(1-e^{-st}).
  \end{align}
\end{proposition}
\begin{remark}
  If $\mathbf x(0)$ is a Poisson$(\mu)$ distribution then substituting
  into \eqref{solution to cds} we see that $\mathbf x(t)$ is a Poisson
  distribution with parameter $\lambda/s+e^{-st}(\mu -\lambda/s)$.  In
  other words, just as for the discrete dynamical system considered by
  Haigh, vectors of Poisson weights are mapped to vectors of Poisson
  weights.  In particular $\pi := \mbox{Poisson}(\lambda/s)$ is once
  again a stationary point of the system.  Moreover, this proposition
  shows that for any vector $\mathbf x(0)$ with $x_0(0)>0$, the
  solution converges to this stationary point. The corresponding
  convergence result in the discrete case is established in Maia et
  al.  (2003). More generally, if $k^*$ is the smallest value of $k$
  for which $x_k(0)>0$ then the solution will converge to
  $(\pi_{k-k^*})_{k=0,1,2,\ldots}$.
\end{remark}

\begin{proof}[Proof of Proposition~\ref{prop1}]
Using \eqref{cds} we have
\begin{align*}
  \frac{d}{dt} & \log \Big( \frac{1}{x_0} \sum_{k=0}^\infty x_k
  e^{-\xi k}\Big) = \frac{x_0}{\sum_{k=0}^\infty x_k e^{-\xi k}} \Big(
  - \frac{\sum_{k=0}^\infty x_k e^{-\xi k}}{x_0^2} \dot x_0 +
  \frac{1}{x_0} \sum_{k=0}^\infty \dot x_ke^{-\xi k}\Big) \\
  & = -s \sum_{j=0}^\infty j x_j + \lambda - s\frac{\sum_{k=0}^\infty
    k x_k e^{-\xi k}}{\sum_{k=0}^\infty x_k e^{-\xi k}} +
  s\sum_{j=0}^\infty j x_j \\ & \qquad \qquad \qquad \qquad +
  \frac{\lambda}{\sum_{k=0}^\infty x_k e^{-\xi k}} \Big(e^{-\xi}
  \sum_{k=1}^\infty x_{k-1} e^{-\xi (k-1)} - \sum_{k=0}^\infty x_k
  e^{-\xi k}\Big) \\
  & = s \frac{d}{d\xi} \log\Big( \sum_{k=0}^\infty x_k e^{-\xi
    k}\Big) + \lambda e^{-\xi}.
\end{align*}
Thus, \eqref{eq:siDef} gives
\begin{align*}
  \frac{d}{dt} \sum_{k=0}^\infty \kappa_k \frac{(-\xi)^k}{k!} = -s
  \sum_{k=0}^\infty \kappa_{k+1} \frac{(-\xi)^{k}}{k!} + \lambda
  e^{-\xi}.
\end{align*}
Comparing coefficients in the last equation we obtain
\begin{align*}
  \dot \kappa_k = -s \kappa_{k+1} + \lambda,\qquad \qquad k=0,1,\ldots.
\end{align*}
This linear system can readily be solved. We write
$$ D:=(\delta_{i+1,j})_{i,j=0,1,2,\ldots},\qquad \underline 1 =
(1,1,\ldots),$$ so that
\begin{align}\label{eq:smLS}
  \dot{\underline \kappa}^\top = -sD \underline \kappa^\top +
  \lambda\underline 1^\top.
\end{align}
Since
\begin{align*}
  (e^{-Dst})_{ij}= \begin{cases} \frac{(-st)^{j-i}}{(j-i)!} & j\geq
    i\\0 & \text{otherwise,} \end{cases}
\end{align*}
the linear system \eqref{eq:smLS} is solved by
$$ \underline \kappa(t)^\top = e^{-Dst} \underline \kappa(0)^\top 
+ \lambda\int_0^t e^{-Dsu}\underline 1 du = e^{-Dst} \underline
\kappa(0)^\top + \frac{\lambda}{s} ( 1 - e^{-st}) \underline 1^\top.$$
\end{proof}
\begin{remark}\label{phase1} With the initial condition $\mathbf x(0) := \tilde \pi$ given by \eqref{pitilde}, equations \eqref{solutionx0} and \eqref{solutionM} become
\begin{align}\label{eq:p1}
x_0(t) =  e^{-\theta} \frac{\theta e^{-st}}{1-e^{-\theta e^{-st}}}
\end{align} 
and
\begin{align}\label{eq:p2}
\kappa_1(t) = \theta -1 + \frac{\theta e^{-st}}{e^{\theta e^{-st}}-1}.
\end{align} 
At time 
\begin{align}\label{tau}
\tau := \frac {\log \theta}s,
\end{align}
we have
$x_0(\tau)=e^{-\theta}\frac 1{1-e^{-1}} \approx 1.6 \pi_0$. Comparing with
\S\ref{haigh} we see that in our continuous time setting $\tau$ is precisely
the counterpart of the time proposed by Haigh as the end of `phase one'.
\end{remark}

In \S\ref{one-dimensional diffusion} our prediction for $M_1 (\mathbf
Y)$ given $Y_0$ will require the value of $M_1(\mathbf y(\tau))$ for $
\mathbf y$ solving \eqref{cds} when started from a Poisson profile
approximation. This is the purpose of the next proposition.

\begin{proposition}\label{p:relax}
  For $y_0 \in (0,1)$, let $\mathbf y(t)$ be the solution of
  \eqref{cds} with the initial state $\mathbf y(0) :=\Pi(y_0)$ defined
  in \eqref{PPA}, and let $\tau$ be Haigh's relaxation time defined in
  \eqref{tau}. Then for $A\ge 0$ with $\eta := \theta^{1-A}$
%
  \begin{align*}
    M_1(\mathbf y(A\tau)) = \theta + \frac{\eta}{e^\eta-1}\Big( 1 -
    \frac{y_0(A\tau)}{\pi_0 }\Big).
  \end{align*}
\end{proposition}

\begin{proof}
  Since
  \begin{align*} \sum_{k=0}^\infty \pi_k e^{-\xi k} = \exp\big(
    -\theta(1-e^{-\xi})\big) = \pi_0^{1-e^{-\xi}},
  \end{align*}
  we have
  \begin{align*} 
    \sum_{k=0}^\infty y_k e^{-\xi k}\Big|_{\xi = sA\tau} = y_0 +
    \frac{1-y_0}{1-\pi_0} \pi_0\big( e^{\theta e^{-\xi}}
    -1\big)\Big|_{\xi = sA\tau} = y_0 + \frac{1-y_0}{1-\pi_0} \pi_0
    (e^\eta-1)
  \end{align*}
  and
  \begin{align*}
    -\frac{\partial}{\partial \xi} \sum_{k=0}^\infty y_k e^{-\xi
      k}\Big|_{\xi = sA\tau} = \frac{1-y_0}{1-\pi_0}
    \pi_0^{1-e^{-\xi}}\theta e^{-\xi}\Big|_{\xi = sA\tau} =
    \frac{1-y_0}{1-\pi_0}\pi_0 e^\eta \eta.
  \end{align*}
  Using the solution \eqref{solutionx0} and \eqref{solutionM} and
  $y_0(0)=y_0$
  \begin{align}\notag
    y_0(A\tau) &=
    y_0\frac{\pi_0 e^\eta}{y_0 + \frac{1-y_0}{1-\pi_0}\pi_0(e^\eta-1)} \\
    \label{eq:relPPA4}
    & = y_0\frac{\pi_0 e^\eta(1-\pi_0)}{y_0(1-\pi_0e^\eta) +
      \pi_0(e^\eta-1)},\\
    \notag M_1(\mathbf y(A\tau)) & = \frac{\frac{1-y_0}{1-\pi_0}\pi_0 e^\eta \eta
    }{y_0 +
      \frac{1-y_0}{1-\pi_0} \pi_0(e^\eta-1)} + \theta - \eta  \\
    & = \theta + \eta \frac{\pi_0-y_0}{y_0(1-\pi_0e^\eta) +
      \pi_0(e^\eta-1)}.\label{eq:relPPA5}
  \end{align}
  From \eqref{eq:relPPA4}, 
  \begin{align*}
    y_0 = \frac{y_0(A\tau) \pi_0(e^\eta-1)}{\pi_0 e^\eta(1-\pi_0) -
      y_0(A\tau)(1-\pi_0 e^\eta)} 
  \end{align*}
  and thus
  \begin{align*}
    &\pi_0-y_0 = 
       \frac{\pi_0e^\eta(\pi_0 -
      y_0(A\tau))(1-\pi_0)}{\pi_0e^\eta(1-\pi_0) -
      y_0(A\tau)(1-\pi_0e^\eta)},
 \\
 &y_0(1-\pi_0 e^\eta) + \pi_0(e^\eta-1) = \pi_0(e^\eta-1)
    \frac{\pi_0e^\eta(1-\pi_0)}{\pi_0e^\eta(1-\pi_0) -
      y_0(A\tau)(1-\pi_0e^\eta)}.
  \end{align*}
  Plugging the last two equations into \eqref{eq:relPPA5} we find
  \begin{align*}
    M_1(\mathbf y(A\tau)) & = \theta + \frac{\eta}{e^\eta-1}\Big( 1 -
    \frac{y_0(A\tau)}{\pi_0 }\Big).
  \end{align*}
\end{proof}

\section{One dimensional diffusion approximations}
\label{one-dimensional diffusion}
Recall from \S\ref{Fleming-Viot} that in our Fleming-Viot model the 
frequency $Y_0$ of the best class follows
\begin{align}\label{bestclass}
dY_0 =\Big(
s M_1({\bf Y})-\lambda \Big)Y_0 dt + \sqrt{\frac 1N Y_0(1-Y_0)} \, dW_0,
\end{align}
where $W_0$ is a standard Wiener process.
The system of 
equations \eqref{eq:FV} is too complex for us to be able to find an 
explicit expression for $M_1({\bf Y})$, which depends on the whole
vector ${\bf Y}$ of class sizes.  Instead we seek a good
approximation of $M_1$ given $Y_0$.  Substituting this into
equation~(\ref{bestclass}) will then yield a one-dimensional 
diffusion which we use as an approximation for the size of the best class.
Of course this assumption of a functional dependence between $Y_0$ and 
$M_1$ is a weakness of the one-dimensional diffusion approximation, 
but simulations show that there is
a substantial correlation between $Y_0$ and $M_1$,
see, for example, Figure~\ref{fig2}.

\begin{figure}
\includegraphics[width=4in]{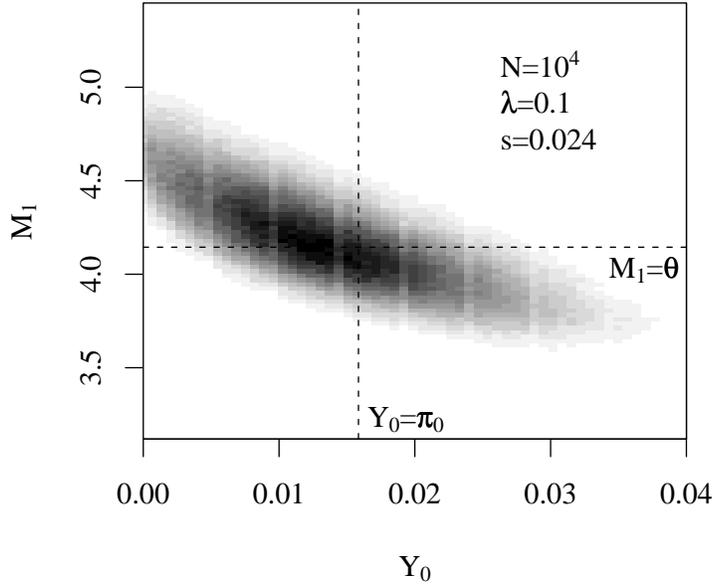}

\caption{\label{fig2}Using simulations (see also \S\ref{simulations})
  we plot $(Y_0,M_1)$. There is a good fit to a linear relationship
  between $Y_0$ and $M_1$. Note that $\gamma=0.6$ in the figure. }
\end{figure}

To understand our approach to finding a map $Y_0\mapsto M_1$, 
recall as a first step Haigh's approximation that immediately
after a click of the ratchet the profile has the form \eqref{pitilde}. 
The reasoning is as follows. Deviations of $\mathbf Y$ from a Poisson
profile can only be due to the randomness arising from resampling
in a finite population.  Since resampling has no
tendency to increase or decrease the frequency of a given class, the
average profile immediately after a click of the ratchet is
approximated by the state where $\pi_0$ is distributed evenly over all
other
classes according to their equilibrium frequencies.  During his short
`phase one', Haigh then allows this profile to `relax' through the
action of the discrete dynamical system~(\ref{DDR}) and it
is the mean fitness in the population after this short relaxation time
which determines the behaviour of the best class during `phase two'.

A natural next step in extending this argument
is to suppose that also {\em in between} click times the resampling
distributes the mass $\pi_0-Y_0$ evenly on all other classes. In other words,
given $Y_0$, approximate the state of the system by $\mathbf Y =
\Pi(Y_0)$ given by \eqref{PPA}.

Of course in reality the dynamical system interacts with the
resampling as it tries to restore the system to its Poisson
equilibrium.  If this restoring force is strong, just as in Haigh's
approach one estimated mean fitness during phase two from the
`relaxed' profile, so here one should approximate the mean fitness
$M_1$ not from the PPA, but from states which arise by evolving the
PPA using the dynamical system for a certain amount of time.  We call
the resulting states \emph{relaxed Poisson Profile approximations} or
RPPA. There are three different parameter regimes with which we shall
be concerned. Each corresponds to a different value of $\eta$ in
the functional relationship
\begin{align} \label{eq:functRel}
  M_1 = \theta + \frac{\eta}{e^\eta-1}\Big( 1 - \frac{Y_0}{\pi_0}\Big).
\end{align}
of Proposition \ref{p:relax}. These can be distinguished as follows:

\begin{subequations}
  \label{eq:three}
  \begin{align}
  \label{eq:threea}
    &A \text{ small},\quad &&\eta \approx \theta,\quad &&
    M_1\approx \frac{\theta}{1-\pi_0}( 1 - Y_0),\\
    \label{eq:threeb}
    &A = 1 ,\quad &&\eta = 1,\quad && M_1 \approx \theta + 0.58 \Big(
    1 - \frac{Y_0}{\pi_0}\Big),\\
    \label{eq:threec}
    &A \text{ large},\quad &&\eta \approx 0,\quad && M_1 \approx \theta
    + \Big( 1 - \frac{Y_0}{\pi_0}\Big)
  \end{align}
\end{subequations}
The resulting maps $Y_0\mapsto M_1$ are plotted in Figure
\ref{fig:three}. Observe that for consistency, $M_1$ has to increase,
on average, by 1 during one click of the ratchet.

\begin{figure}
   \includegraphics[width=11cm]{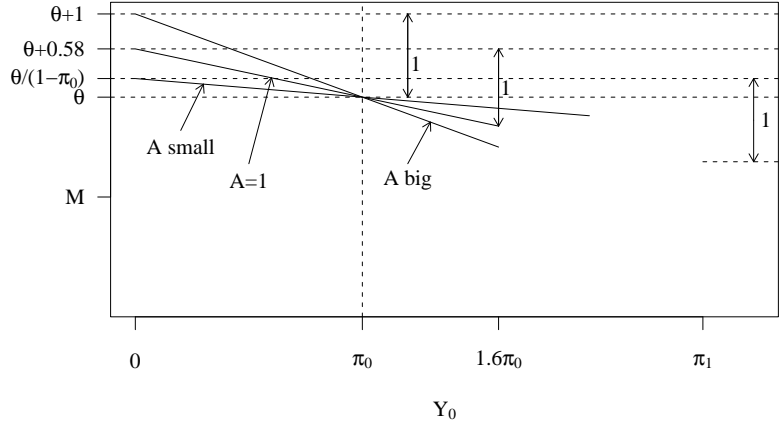}

   \caption{\label{fig:three}Since simulations show a strong
     correlation between the first moment $M_1$ and $Y_0$, we use
     \eqref{eq:threea}-\eqref{eq:threec} to predict $M_1$ from $Y_0$
     depending on the model parameters.}
\end{figure}

Finally, before we can apply our one-dimensional diffusion
approximation we must choose a starting value for $Y_0$ following
equation~(\ref{bestclass}).  For $A$ large, the system is already
close to its new equilibrium at the time of a click and so we take
$Y_0=\pi_0$.  

For $A=1$, at the time of the click we observe a state which has
relaxed for time $\tau$ from a state of the form $\tilde \pi$ from
\eqref{pitilde}. We computed in Remark \ref{phase1} that such a state
comes with $Y_0=1.6 \pi_0$.

For small values of $A$, observe that the profile of the population
immediately after a click is approximately $\tilde\pi$ from
\eqref{pitilde}. Since $\tilde\pi$ is not a state of the form
$\Pi(y_0)$ the arguments that led to Proposition \ref{p:relax} do not
apply. Instead we follow Haigh in dividing the time between clicks
into two phases.  Consider first `phase one'.  Recall from the
dynamical system that
\begin{align*}
  dY_0 = (sM_1 - \lambda)Y_0 \, dt.
\end{align*}
We write
\begin{align*} (sM_1 - \lambda)Y_0 = c\ (\pi_0-Y_0),\end{align*} where
$c$ (like $Y_0$ and $M_1$) depends on $r=\theta e^{-st}$.  Starting in
$\mathbf Y(0) = \tilde\pi$ we have from \eqref{eq:p1} and
\eqref{eq:p2}
\begin{align*}
  Y_0(t) &= e^{-\theta} \frac{\theta e^{-st}}{1-\exp(-\theta e^{-st})}
  = e^{-\theta}\frac{r}{1-e^{-r}},\\ 
  M_1(t) & = \theta - 1 + \frac{\theta e^{-st}}{\exp(\theta
    e^{-st})-1} = \theta-1+\frac{r}{e^{r}-1}.
\end{align*}
We compute
\begin{align*}
  \frac cs = \frac{(M_1 -\theta)Y_0}{ \pi_0-Y_0} = \frac{1 -
    \frac{r}{e^r-1}}{1 - \frac{1-e^{-r}}{r}} =
  \frac{r(1-e^{-r})-r^2e^{-r}}{r(1-e^{-r}) - (1-e^{-r})^2}.
\end{align*}
It can be checked that this expression lies between $1$ and $1.25$ for
all $r > 0$ which suggests that the size of the best class in the
initial phase after a click is reasonably described by the dynamics
\begin{align}\label{ph1stoch}
 dY_0 = s\ (\pi_0-Y_0)\, dt + \sqrt{\frac 1N Y_0} \, dW_0
\end{align}
started from $\frac{\pi_1}{(1-\pi_0)}$.  We allow $Y_0$ to evolve
according to equation~(\ref{ph1stoch}) until it reaches $1.6\pi_0$,
say, and then use our estimate of $M_1$ from equation~(\ref{eq:three})
to estimate the evolution of $Y_0$ during the (longer) `phase two'.
\bigskip

We assume that states of the ratchet are RPPAs, i.e., Poisson profile
approximations \eqref{PPA} which are relaxed for time $A\tau$, where
$\tau = \tfrac 1s \log\theta$, which leads to the functional
relationship \eqref{eq:functRel}. Consequently, we suggest that
\eqref{bestclass} is approximated by the `mean reversion' dynamics
\begin{align}\label{fellog2}
  dY_0 = s \frac{\eta}{e^\eta-1} \big( 1 - \tfrac{Y_0}{\pi_0}\big) Y_0
  dt + \sqrt{\frac 1N Y_0} \, dW_0,
\end{align}
with $\eta=\theta^{1-A}$, where we have used a Feller noise instead of
the Wright-Fisher term in \eqref{bestclass}.  In other words, $Y_0$ is
a {\em Feller branching diffusion with logistic growth}.

Using the three regimes from \eqref{eq:three}, we have the
approximations
\begin{subequations}
  \label{eq:three2}
  \begin{align}
  \label{eq:three2a}
  &A \text{ small},\quad &&
  dY_0 = \lambda (\pi_0 - Y_0)Y_0 dt + \sqrt{\frac 1N Y_0} dW, \\
  \label{eq:three2b}
  &A = 1 ,\quad && dY_0 = 0.58 s \Big( 1 - \frac{Y_0}{\pi_0}\Big) Y_0
  dt + \sqrt{\frac 1N Y_0} \, dW_0,\\
  \label{eq:three2c}
  &A \text{ large},\quad && dY_0 = s \Big( 1 - \frac{Y_0}{\pi_0}\Big)
  Y_0 dt + \sqrt{\frac 1N Y_0} \, dW_0,
  \end{align}
\end{subequations}
(where in the first equation we have used that
$\frac{1}{1-\pi_0}\approx 1+\pi_0$ and that $Y_0\pi_0$ is negligible).

An equation similar to ~(\ref{eq:three2b}) was found (by different
means) by Stephan et al (1993) and further discussed in Gordo \&
Charlesworth (2000). Stephan and Kim (2002) analyse
\nocite{StephanKim2002} whether a prefactor of 0.5 or 0.6 in
~(\ref{eq:three2b}) fits better with simulated data. We discuss
the relationship with these papers in detail in \S\ref{discussion}.

The expected time to extinction of a diffusion following
\eqref{fellog2} is readily obtained from a Green function calculation
similar to that in Lambert~(2005).  We refrain from doing this here,
but instead use a scaling argument to identify parameter ranges for
which the ratchet clicks and to give evidence for the rule of thumb
formulated in the introduction.  \nocite{Lambert2005}

\medskip
Consider the rescaling 
$$Z(t)=\frac{1}{\pi_0}Y_0\left(N\pi_0 t\right).$$
For $A$ small equation~(\ref{eq:three2a}) becomes
\begin{align}\notag
dZ & =N\lambda \pi_0^2 (1-Z)Zdt +\sqrt{Z}dW\\
& =(N\lambda)^{1-2\gamma}(1-Z)Zdt
+\sqrt{Z}dW.
\label{rescaled PPA}
\end{align}
For $A=1$ on the other hand we obtain from \eqref{eq:three2b}
\begin{align}\notag
dZ & = 0.58Ns\pi_0(1-Z)Zdt +\sqrt{Z}dW \\
 &= 0.58\frac{1}{\gamma\log(N\lambda)}
(N\lambda)^{1-\gamma}(1-Z)Zdt
+\sqrt{Z}dW.
\label{rescaled SC}
\end{align}
For $A$ large we obtain from \eqref{eq:three2c} the same equation
without the factor of $0.58$.

 From this rescaling we see that  the equation that applies for small $A$, i.e. \eqref{rescaled PPA}, is strongly
mean reverting for $\gamma <1/2$.  Recall that the choice of small $A$
is appropriate when the ratchet is clicking frequently and so this
indicates that frequent clicking simply will not happen for $\gamma <1/2$.  To indicate the boundary between rare and moderate clicking, equation \eqref{rescaled SC} is much more relevant than equation 
\eqref{rescaled PPA}. 
At first sight, equation \eqref{rescaled SC} looks strongly mean reverting for all $\gamma < 1$, which would seem to suggest
that the ratchet will click only exponentially slowly in
$(N\lambda)^{1-\gamma}$. However, the closer $\gamma$ is to one, the larger
the value of $N\lambda$ we must take for this asymptotic regime to
provide a good approximation.  For example, in the table below we
describe parameter combinations for which the coefficient in front of
the mean reversion term in equation~(\ref{rescaled SC}) is at least
five.  We see that for $\gamma < 1/2$ this coefficient is large for most of the reasonable values of $N\lambda$, whereas for  $\gamma > 1/2$ it is rather small over a large range of  $N\lambda$.
\begin{center}
  \begin{tabular}{|c||c|c|c|c|c|c|c|c|} \hline
    \rule[-4mm]{0cm}{1cm}$\gamma$ & $0.3$ & $0.4 $   &$0.5$ & $0.55 $ & $0.6$ & $0.7$ &
    $0.8$ & $0.9$ \\\hline \rule[-4mm]{0cm}{1cm}$N\lambda\ge$ &
    $20 $&
    $10^2 $&$ 9\cdot
    10^2$ & $4\cdot 10^3$ & $2\cdot 10^4$ & $4\cdot 10^6$ & $2\cdot
    10^{11}$ & $8\cdot 10^{26}$ \\\hline
  \end{tabular}
\end{center}
\noindent
Thus, for example, if $\gamma=0.7$ we require $N\lambda$ to be of the
order of $10^6$ in order for the strong mean reversion of
equation~(\ref{rescaled SC}) to be evident.  This is not a value of
$N\lambda$ which will be observed in practice.  Indeed, as a `rule of
thumb', for biologically realistic parameter values, we should expect
the transition from no clicks to a moderate rate of clicks to take
place at around $\gamma=0.5$.

\section{Simulations}
\label{simulations} 
We have argued that the one-dimensional diffusions \eqref{eq:three2}
approximate the frequency in the best class and from this deduced the
\emph{rule of thumb} from \S\ref{intro}.  In this section we use
simulations to test the validity of our arguments.

For a population following the dynamics \eqref{DR}, the $(t+1)$st generation
is formed by multinomial sampling of $N$ individuals with weights
\begin{equation}
\label{multinomial weights}
p_k(t)=\sum_{j=0}^k\frac{x_{k-j}(t)(1-s)^{k-j}}{W(t)}e^{-\lambda}
\frac{\lambda^j}{j!},
\end{equation}
where $W(t)$ is the average fitness in the $t$th generation from
\eqref{eq:fit} and it is this Wright-Fisher model which was
implemented in the simulations.
 \begin{figure}
  \begin{center}
    \vspace{-.3cm}
    \hspace{.5cm} (A) \hspace{5cm} (B)
    \vspace{-.3cm}
  \end{center}
  \begin{center}
    \vspace{-1cm}
    \includegraphics[width=5.5cm]{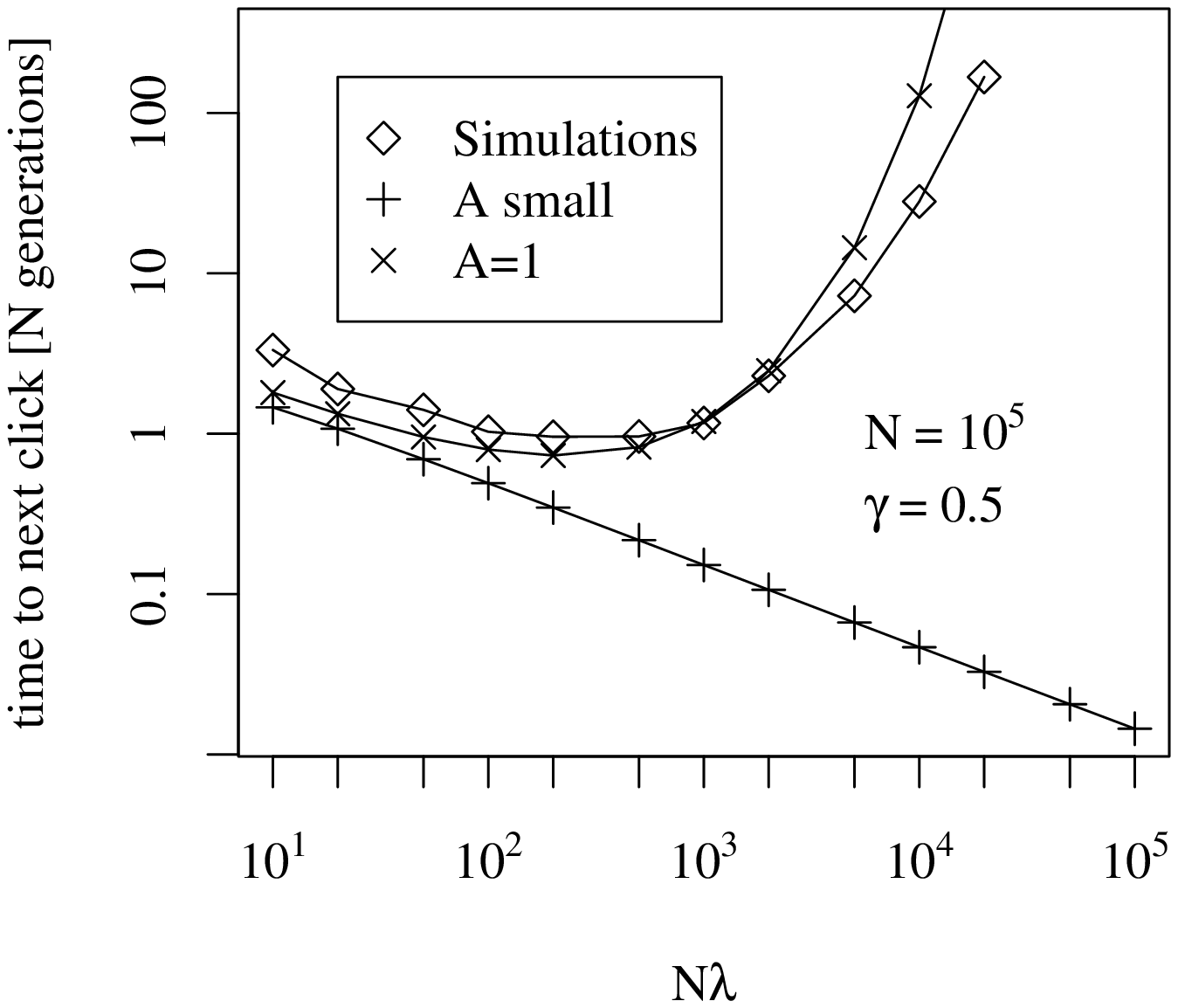}
    \includegraphics[width=5.5cm]{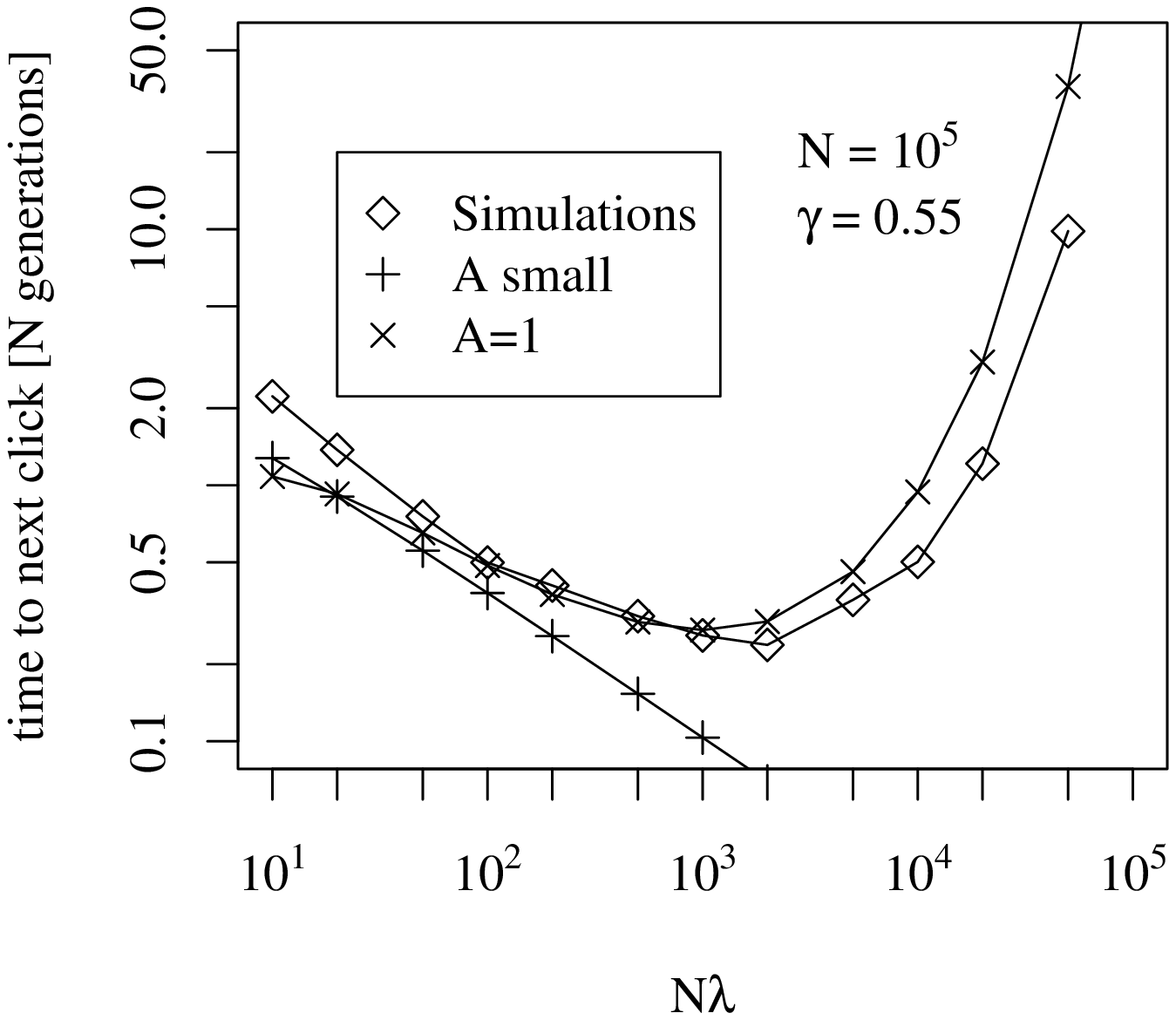}
  \end{center}
  \begin{center}
    \vspace{-.3cm}
    \hspace{.5cm} (C) \hspace{5cm} (D)
    \vspace{-.3cm}
  \end{center}
  \begin{center}
    \vspace{-1cm}
    \includegraphics[width=5.5cm]{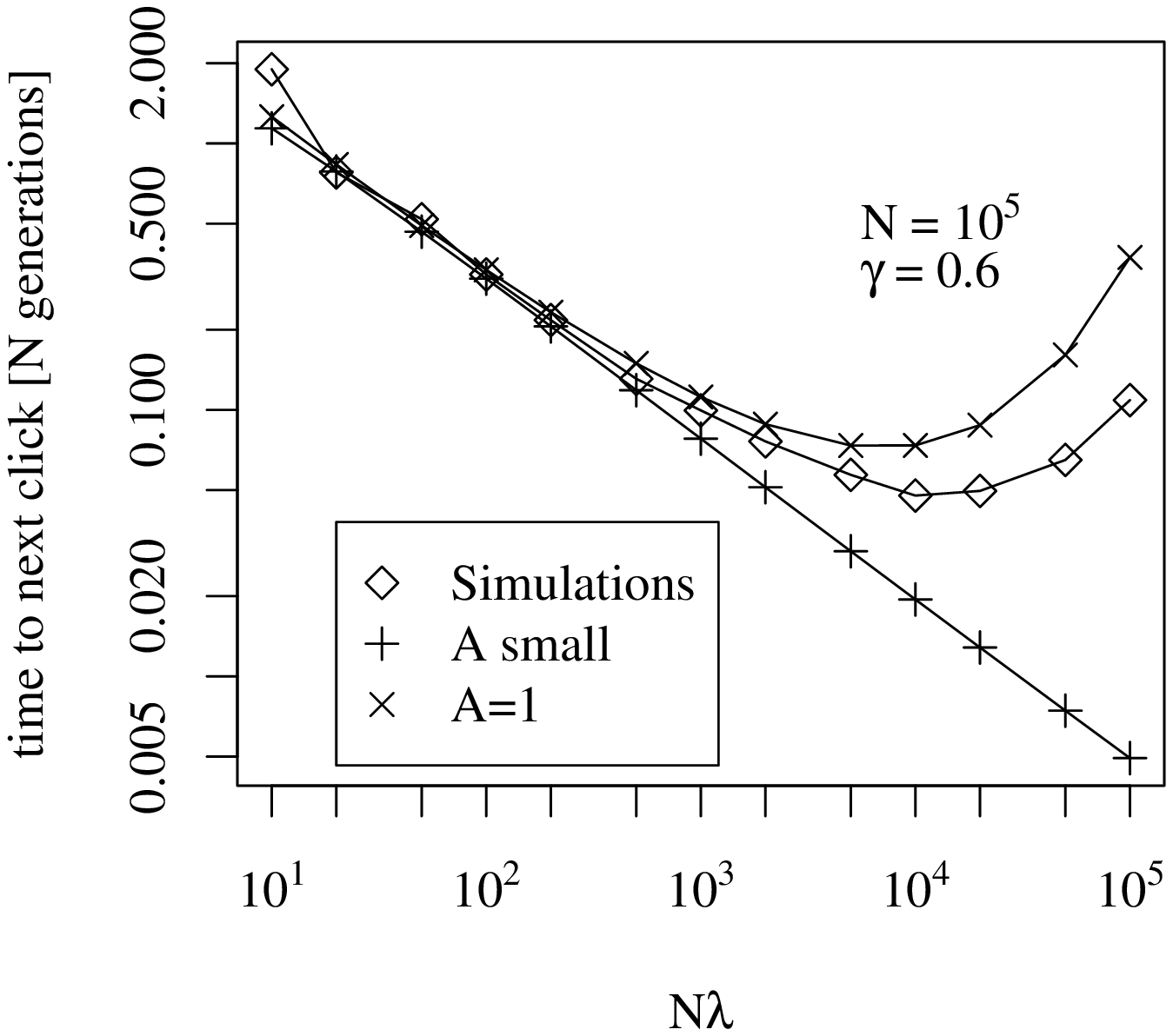}
    \includegraphics[width=5.5cm]{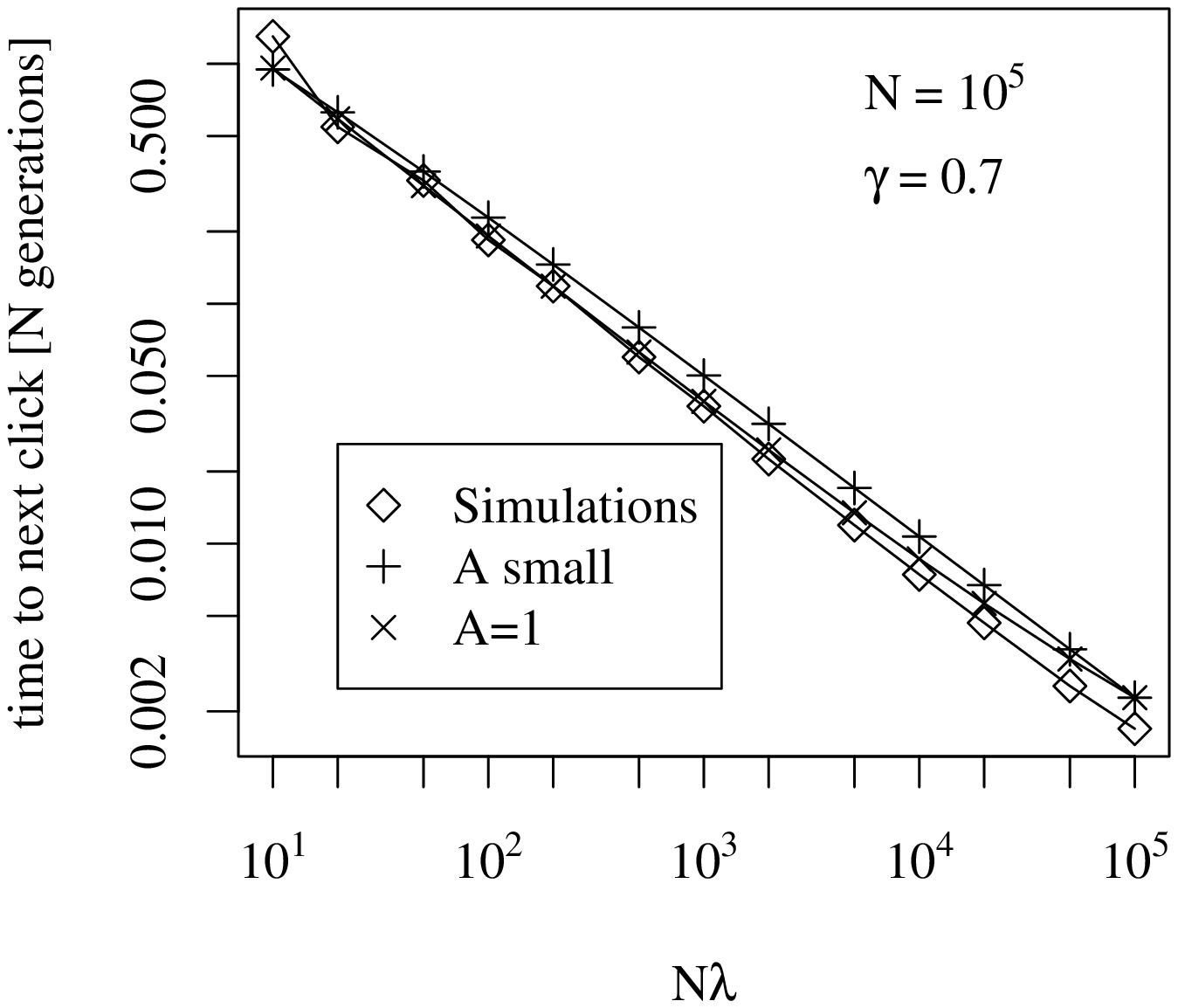}
  \end{center}
  \begin{center}
    \vspace{-.3cm}
    \hspace{.5cm} (E) \hspace{5cm} (F)
    \vspace{-.3cm}
  \end{center}
  \begin{center}
    \vspace{-1cm}
    \includegraphics[width=5.5cm]{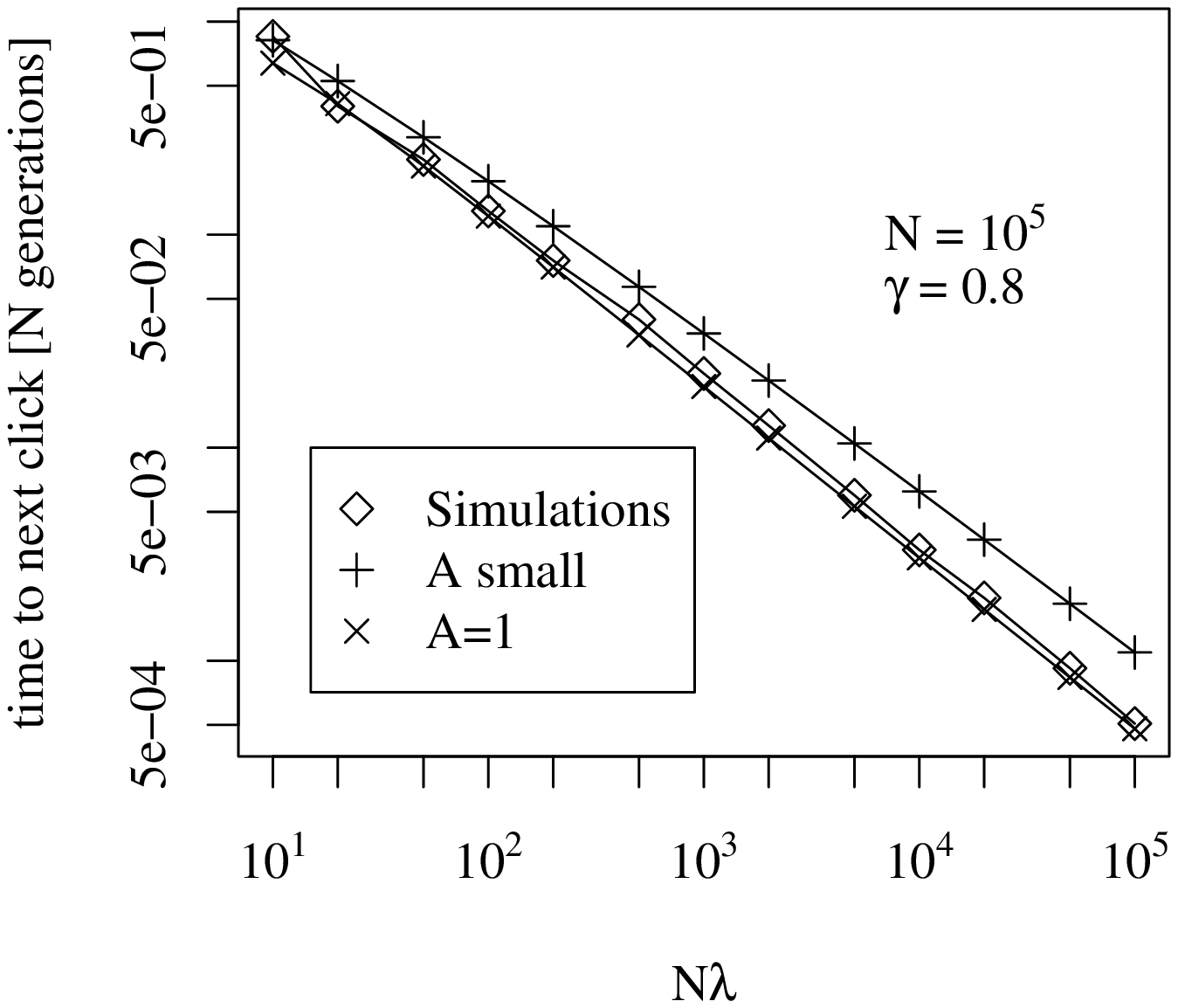}
    \includegraphics[width=5.5cm]{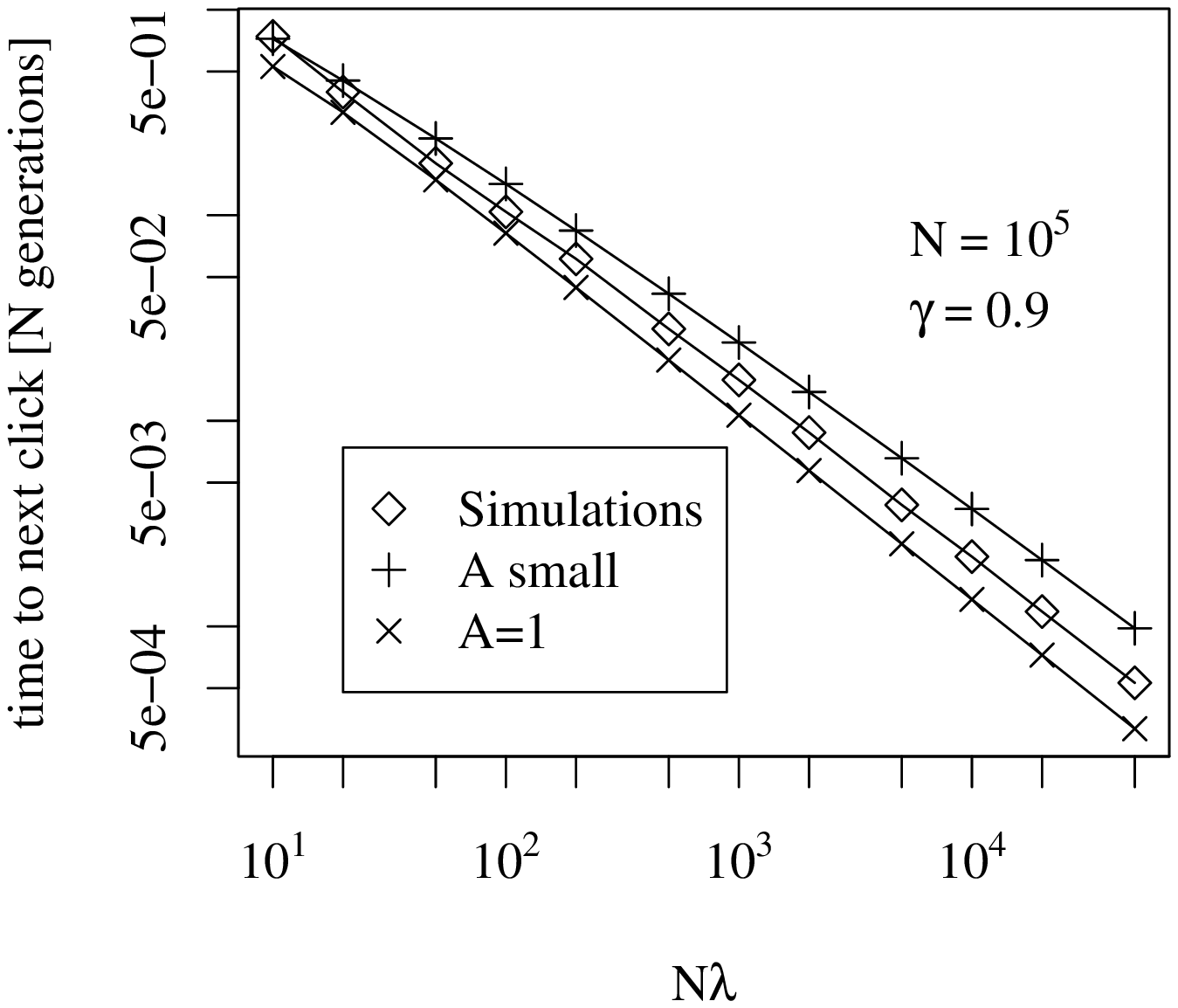}
  \end{center}
  \caption{\label{fig:power} The power law behaviour of the rate of
    the ratchet with respect to $\gamma$ is valid for a large portion
    of the parameter space. (A) For $\gamma= 0.5$ clicks become rare
    and the power law does not apply for $N\lambda>10^3$.  (B), (C)
    For $\gamma=0.55$ and $\gamma=0.6$, we have to explore a larger
    portion of the parameter space in order to see that the power law
    does not apply any more.  (D), (E), (F) For $\gamma\geq 0.7$ we
    never observe a deviation from the power law. For every plot, we
    used $N=10^5$ and simulations ran for $5\cdot10^6$ ($\gamma=0.5$:
    $2\cdot 10^7$) generations for each value of $N\lambda$. }
\end{figure}

To supplement the numerical results of Figure~\ref{fig1} we provide
simulation results for the average time between clicks (where time is
measured in units of $N$ generations) for fixed $N$ and $\gamma$ and
varying $\lambda$; see Figure~\ref{fig:power}. Note that, for fixed
$\gamma$ in equation \eqref{scaling}, $s$ is increasing with $\lambda$
.  We carry out simulations using a population size of $N=10^5$ and
$\lambda$ varying from $10^{-4}$ to $1$. For $\gamma=0.5$ we observe
that the power law behaviour breaks down already for $N\lambda=10^3$
and the diffusion \eqref{eq:three2b} predicts the clicking of the
ratchet sufficiently well. For increasing $\gamma$, the power law
breaks down only for larger values of $N\lambda$. For $\gamma=0.7$, in
our simulations we only observe the power law behaviour but conjecture
that for larger values of $N\lambda$ the power law would break down;
compare with the table above.

\begin{figure}
  \begin{center}
    \hspace{.5cm} (A) \hspace{5.5cm} (B)
  \end{center}
  \begin{center}
    \vspace{-1cm}
    \includegraphics[width=6cm]{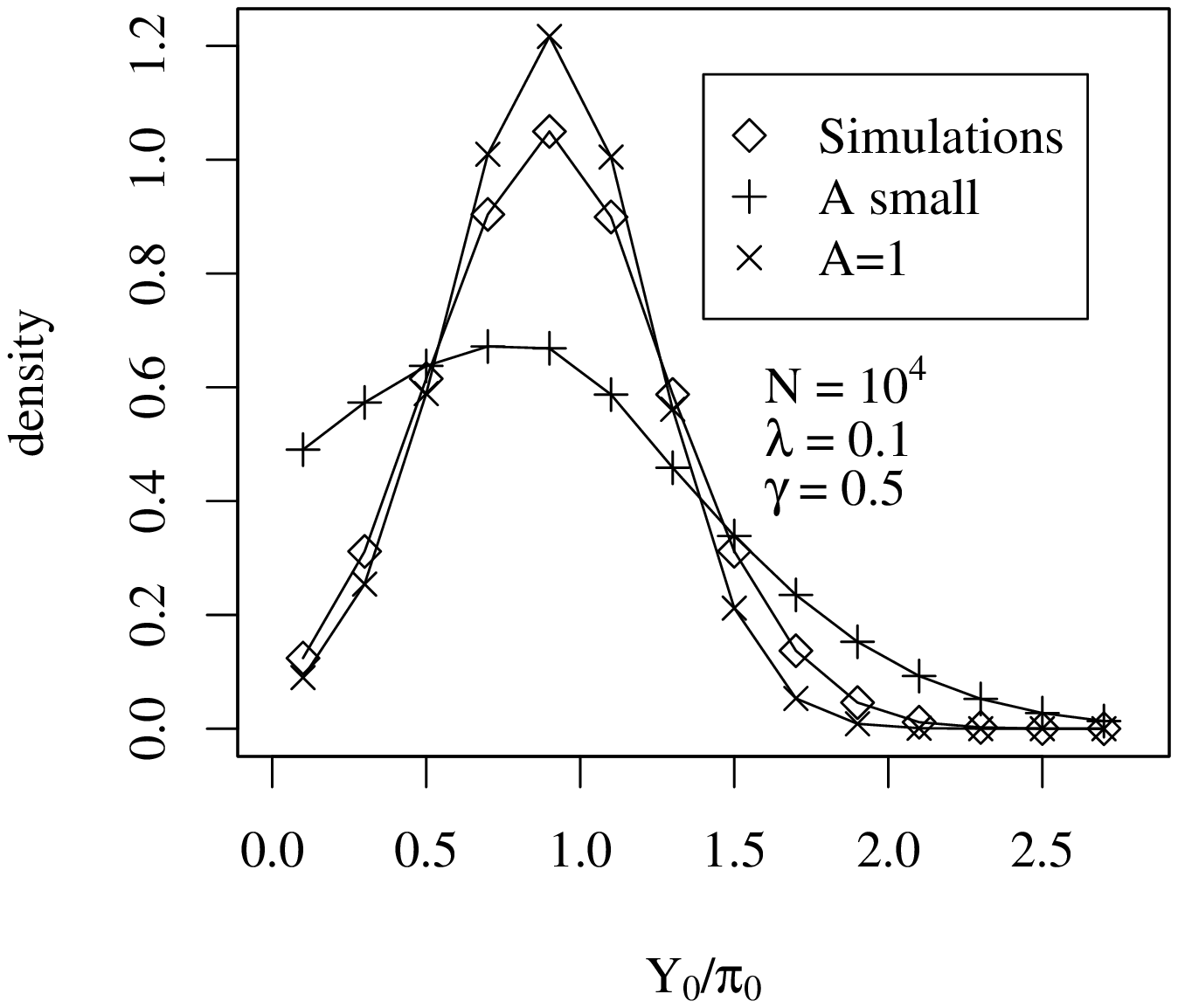}
    \includegraphics[width=6cm]{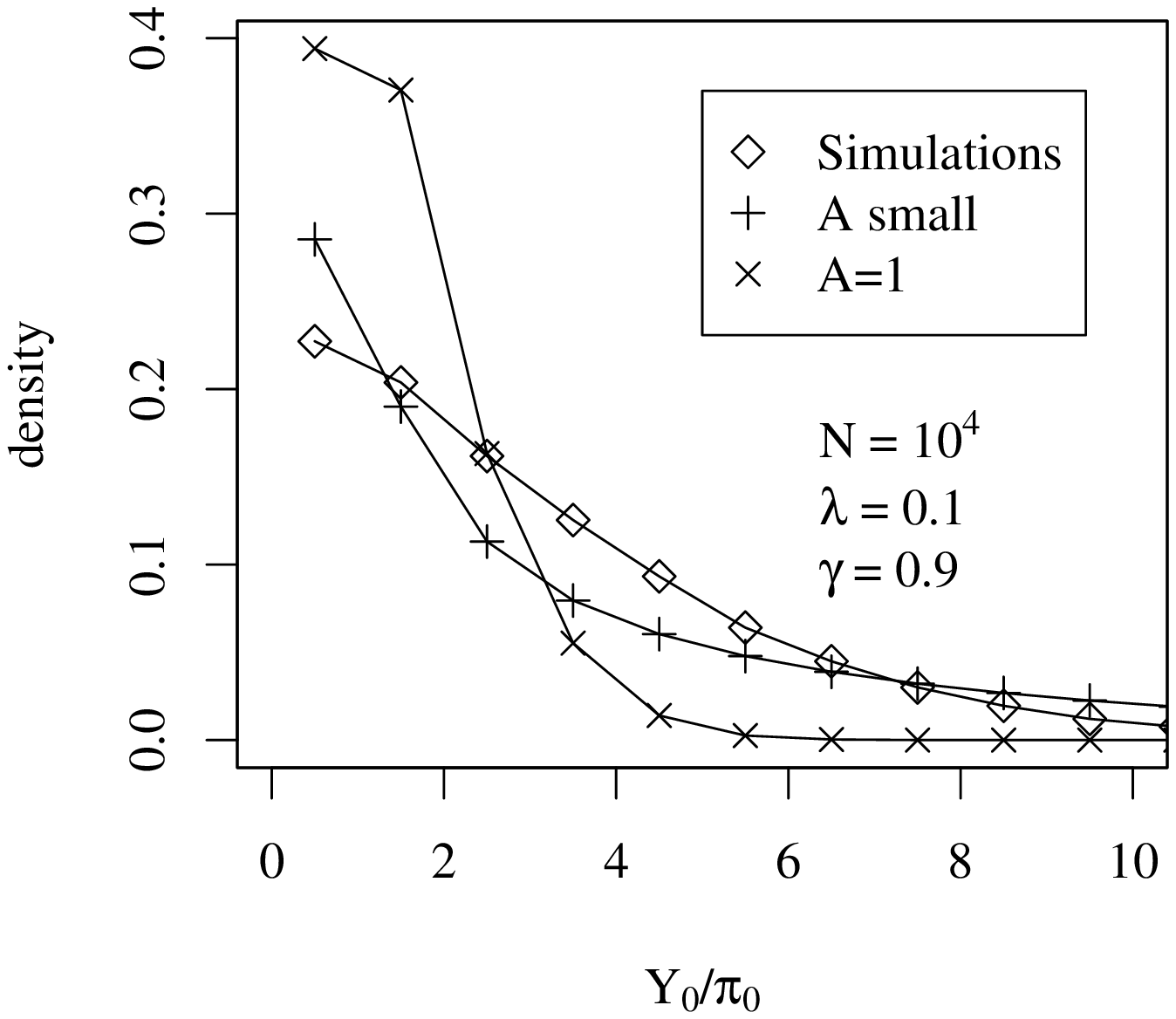}
  \end{center}
  \caption{\label{fig:green}We compare the plots for the occupation
    density of $Y_0$ from the simulations with theoretical curves
    corresponding to the Green functions for the cases of small $A$
    and $A=1$ in \eqref{eq:three2}. (A) If clicks are rare, $A=1$
    produces better results than small $A$. (B) If clicks are
    frequent, the simulated densities of $Y_0$ are better approximated
    by small $A$. Every plot is based on the simulation of $5\cdot
    10^5$ generations.}
\end{figure}

For a finer analysis of which of the equations \eqref{eq:three2} works
best, we study the resulting Green functions numerically; see Figure
\ref{fig:green}. In particular, we record the relative time spent in
some $dY_0$ in simulations and compare this quantity to the
numerically integrated, normalised Green functions given through the
diffusions \eqref{eq:three2a} and \eqref{eq:three2b}. (We do not
consider \eqref{eq:three2c} because it only gives an approximation if
the ratchet clicks rarely.) We see in (A) that for $\gamma=0.5$ not
only does \eqref{eq:three2b} produce better estimates for the average
time between clicks (Figure \ref{fig:power}) but also for the time
spent around some point $y_0$. However, for $\gamma=0.9$, clicks are
more frequent and we expect \eqref{eq:three2a} to provide a better
approximation.  Indeed, although both \eqref{eq:three2a} and
\eqref{eq:three2b} predict the power law behaviour, as (B) shows, the
first equation produces better estimates for the relative amount of
time spent in some $dY_0$.

~

\begin{figure}
  \begin{center}
    \vspace{-.3cm}
    \hspace{.5cm} (A) \hspace{5cm} (B)
    \vspace{-.3cm}
  \end{center}
  \begin{center}
    \vspace{-1cm}
    \includegraphics[width=5.5cm]{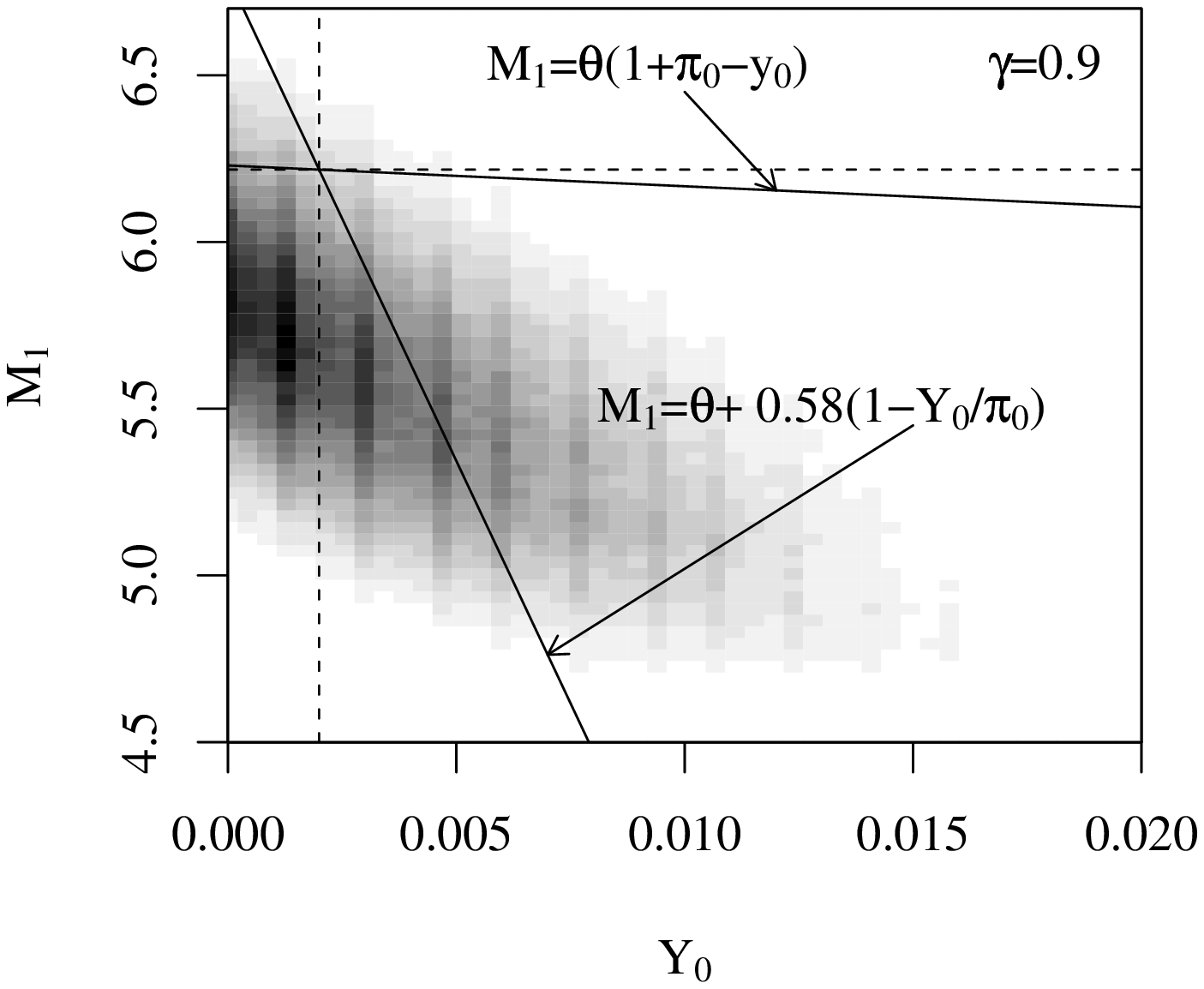}
    \includegraphics[width=5.5cm]{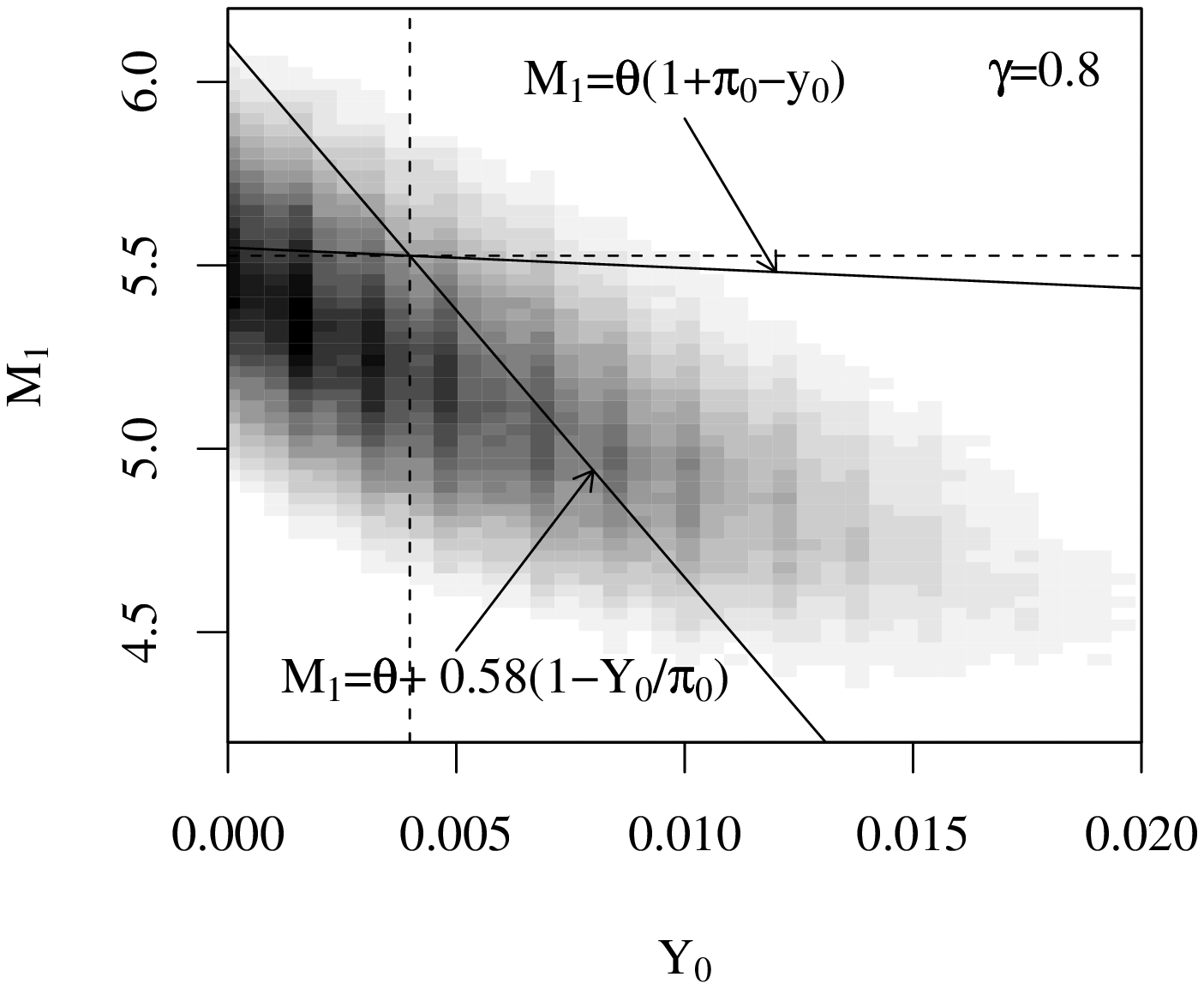}
  \end{center}
  \begin{center}
    \vspace{-.3cm}
    \hspace{.5cm} (C) \hspace{5cm} (D)
    \vspace{-.3cm}
  \end{center}
  \begin{center}
    \vspace{-1cm}
    \includegraphics[width=5.5cm]{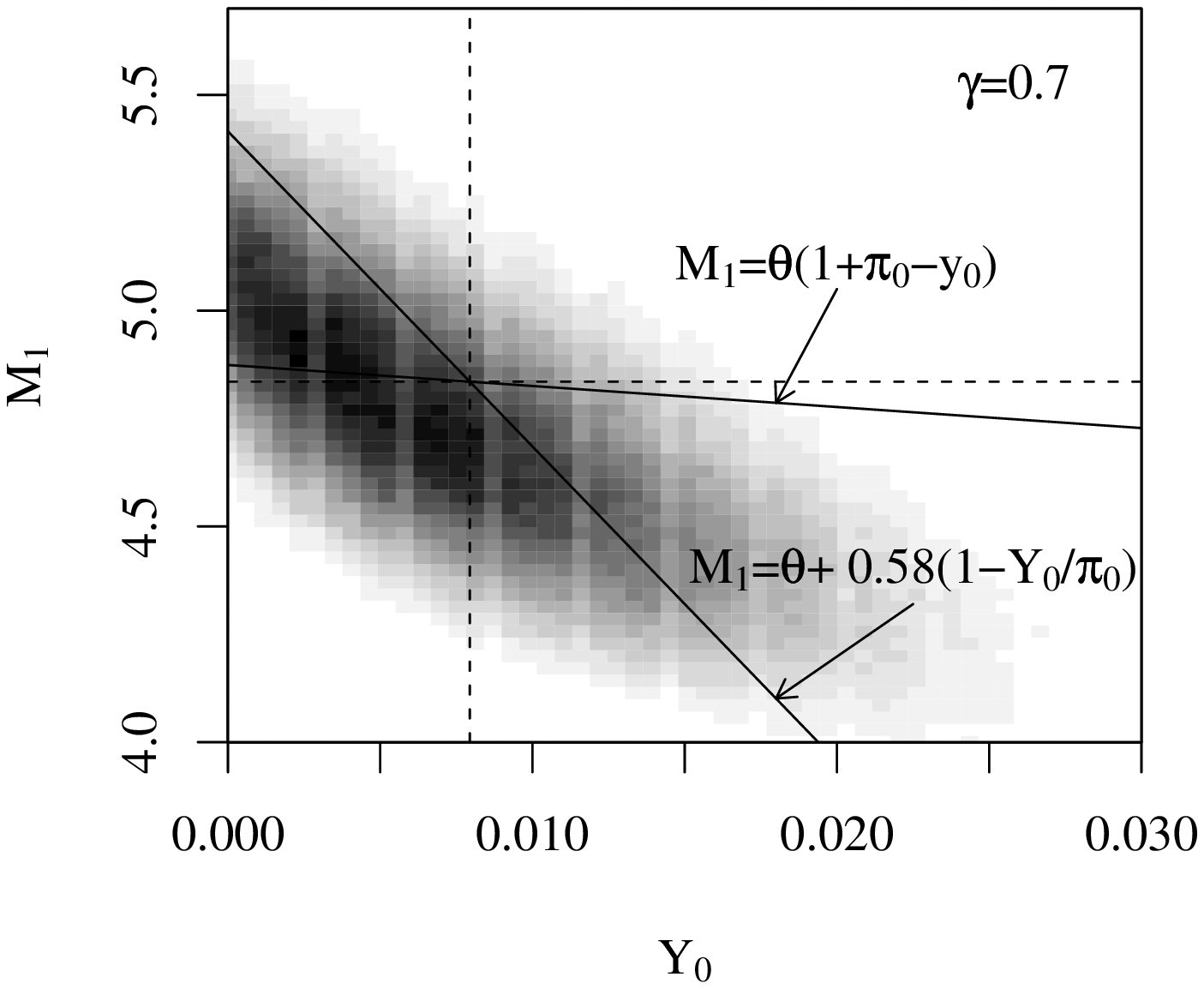}
    \includegraphics[width=5.5cm]{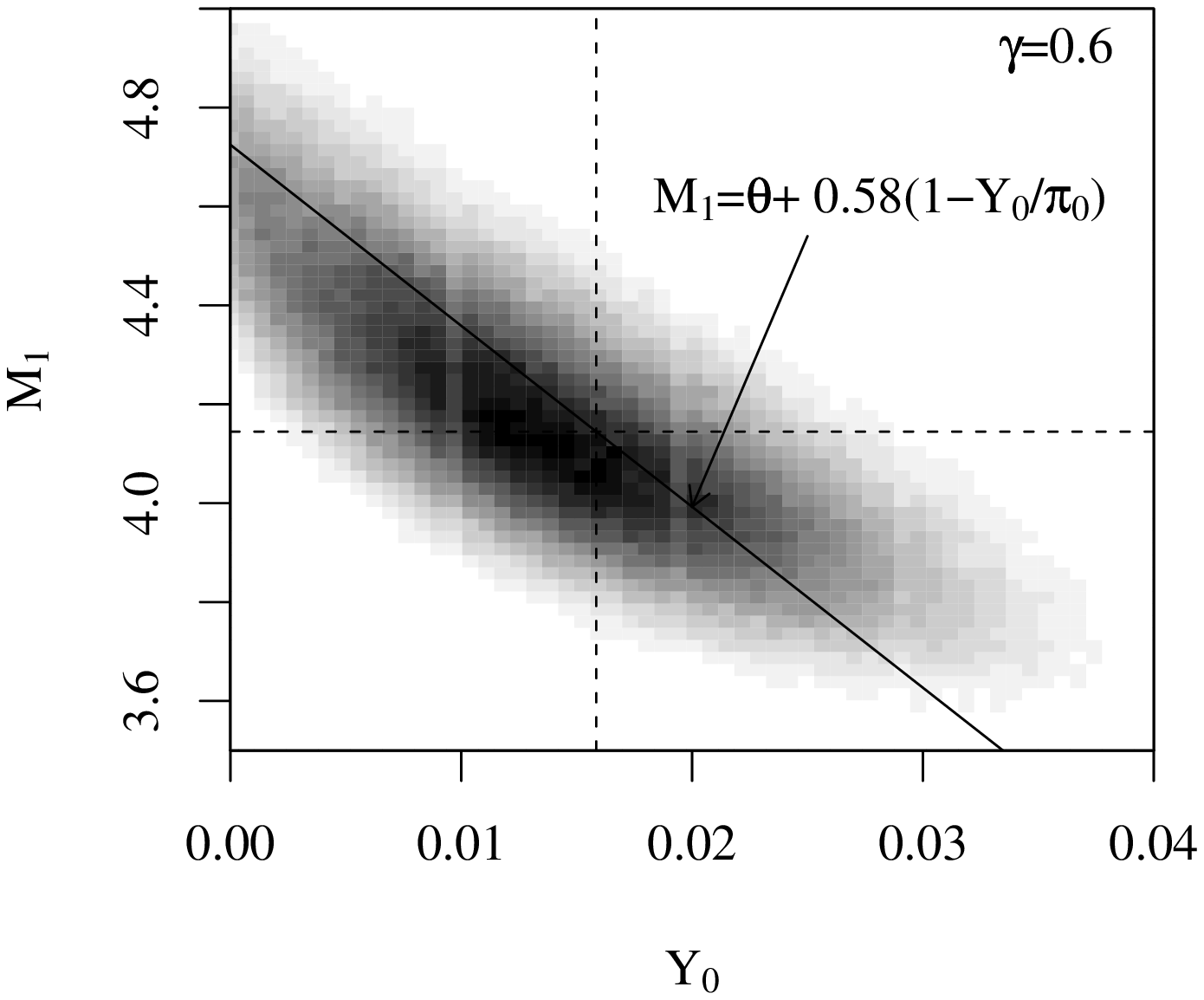}
  \end{center}
  \begin{center}
    \vspace{-.3cm}
    \hspace{.5cm} (E) \hspace{5cm} (F)
    \vspace{-.3cm}
  \end{center}
  \begin{center}
    \vspace{-1cm}
    \includegraphics[width=5.5cm]{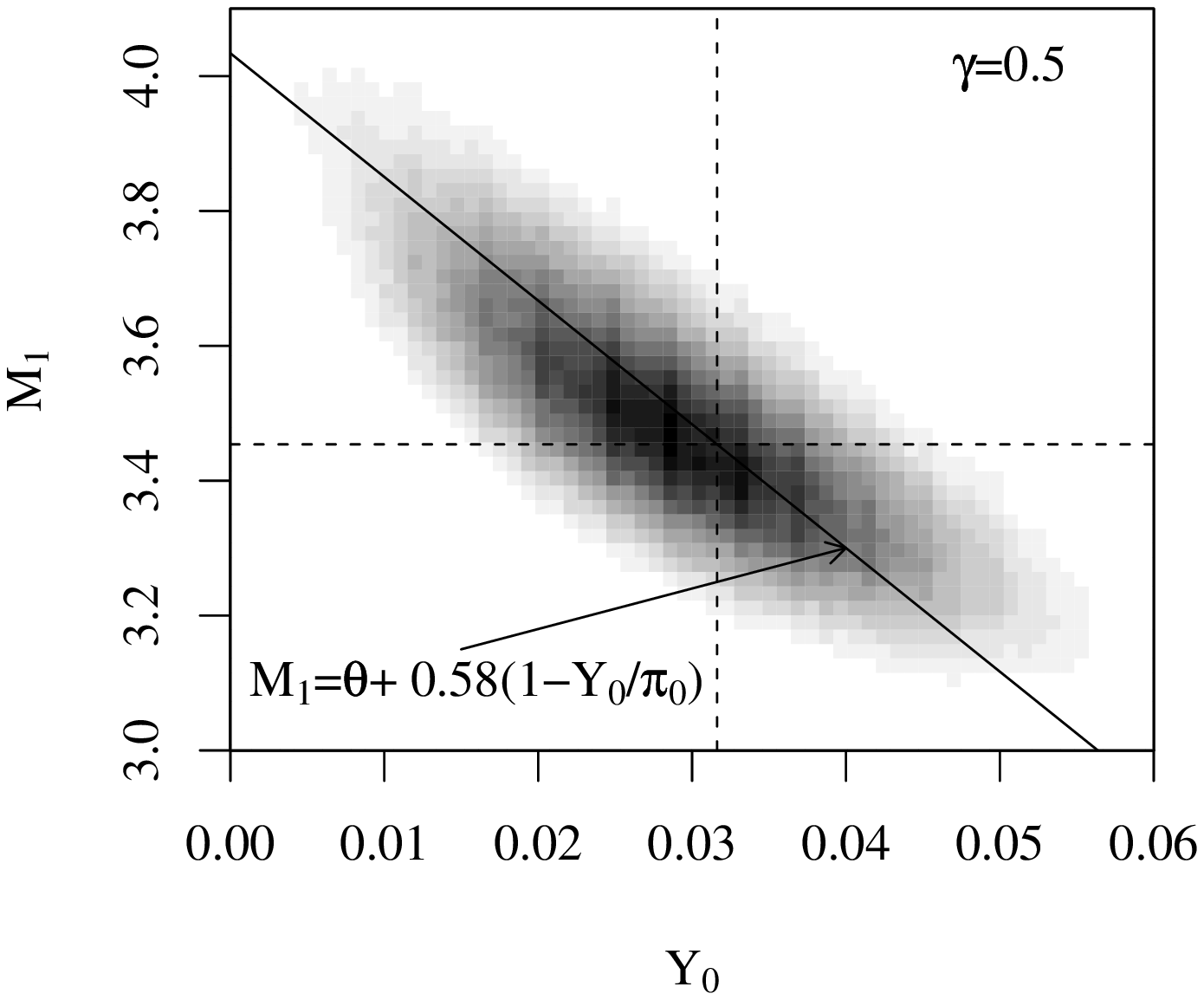}
    \includegraphics[width=5.5cm]{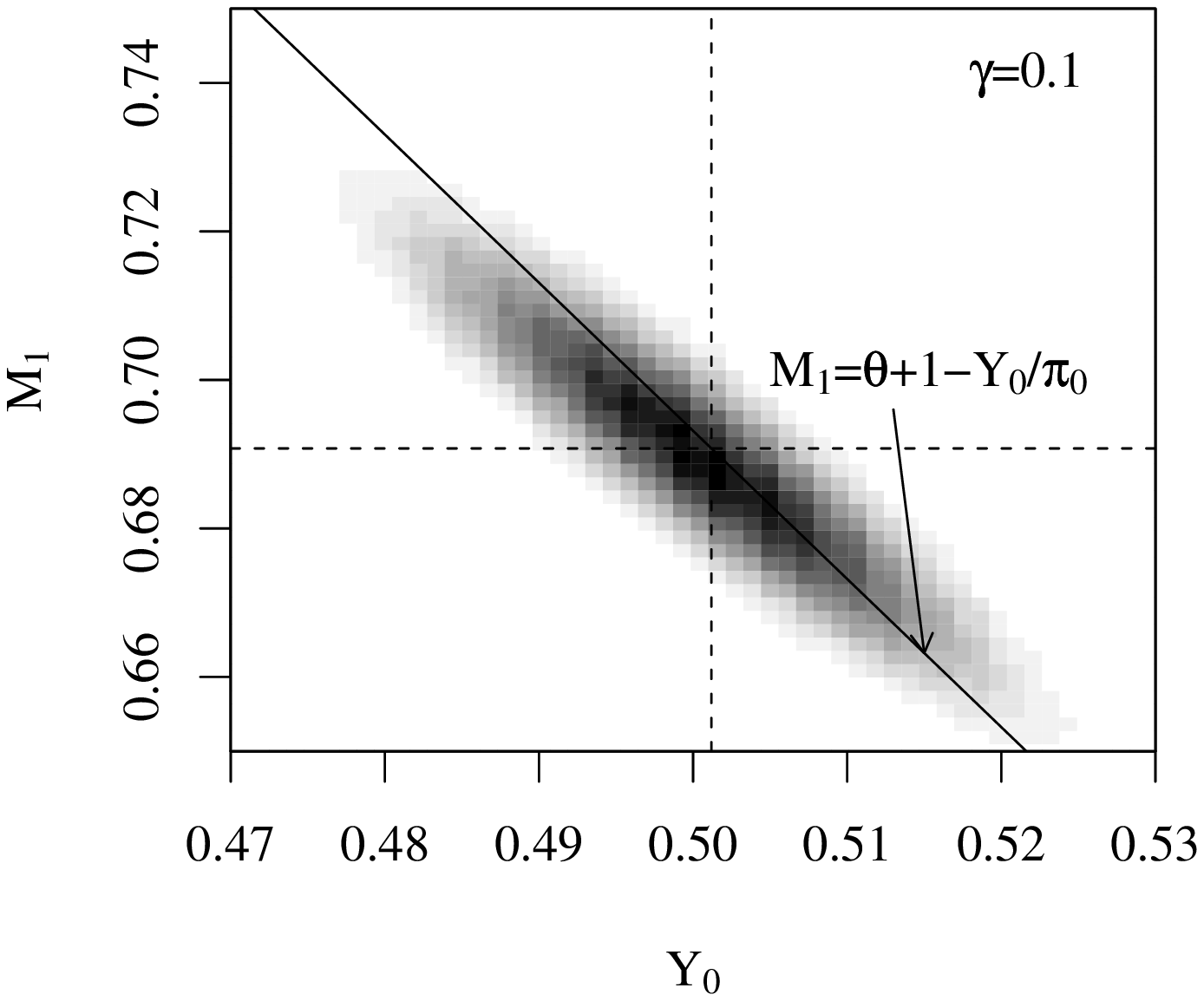}
  \end{center}
  \caption{\label{fig:phaseplane}There are three regimes for the
    relationship between $Y_0$ and $M_1$, as given in
    \eqref{eq:three}. If clicks are frequent, at least the slope of
    the relationship between $Y_0$ and $M_1$ fits roughly to equation
    \eqref{eq:threea}. If clicks occur reasonably often,
    \eqref{eq:threeb} gives a good approximation. If clicks are rare,
    \eqref{eq:threec} gives a reasonable prediction. The plots show
    simulations for different values of $\gamma$. The dashed
    horizontal and vertical lines are $M_1=\theta$ and $Y_0=\pi_0$,
    respectively. For every plot we used $N=10^4, \lambda=0.1$ and
    simulations ran for $10^6$ generations.}
\end{figure}

~

To support our claim that the states observed are relaxed Poisson
Profile Approximations we use a phase-plane analysis; see
Figure~\ref{fig:phaseplane}. At any point in time of a simulation,
values for $Y_0$ and $M_1$ can be observed. The resulting plots
indicate that we can distinguish the three parameter regimes
introduced in \S\ref{one-dimensional diffusion}. In the case of rapid
clicking of the ratchet (so that the states we observe have not
relaxed a lot and thus are approximately of the form $\Pi(y_0)$) we
see in (A), (B) that the system is driven by the restoring force to
$M_1<\theta$.  The reason is that $M_1$ is small at click times and
these are frequent.  However, the slope of the line relating $Y_0$ and
$M_1$ is low, as predicted by \eqref{eq:threea}. (We used
$(1-Y_0)/(1-\pi_0)\approx 1-Y_0+\pi_0$ in the plot here.) For the case
$A=1$ the system spends some time near $Y_0=0$ and thus the ratchet
clicks, but not frequently. So, the dynamical system restores states
partly to equilibrium and we see that the slope given in
\eqref{eq:threeb} gives the most reasonable prediction in (C), (D),
(E).  For rare clicking, i.e. $A$ large, the dynamical system has even
more time and (F) shows that the slope is as predicted by
\eqref{eq:threec}.

\begin{figure}
  \begin{center}
    \hspace{.5cm} (A) \hspace{5.5cm} (B)
  \end{center}
  \begin{center}
    \vspace{-1cm}
    \includegraphics[width=6cm]{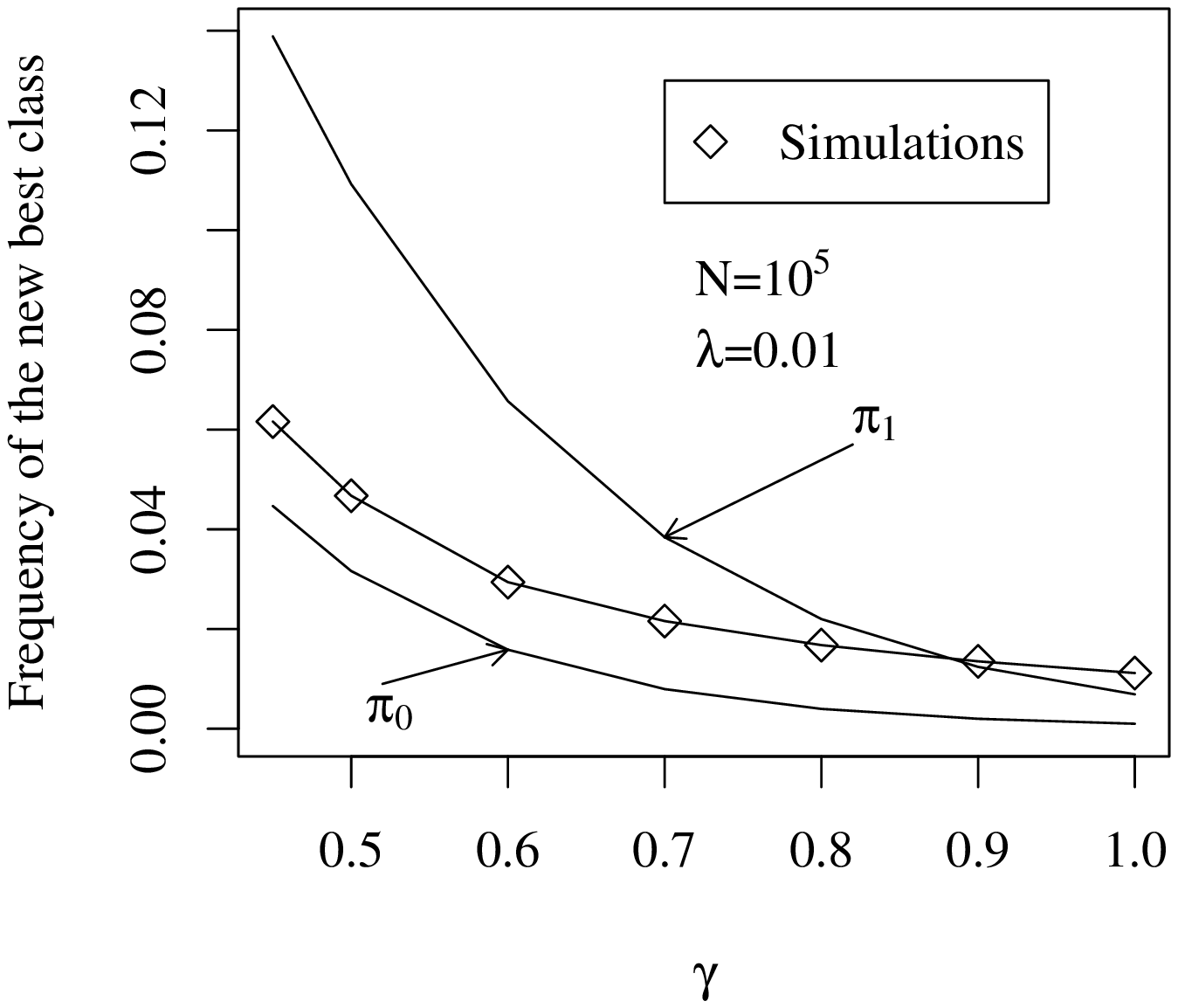}
    \includegraphics[width=6cm]{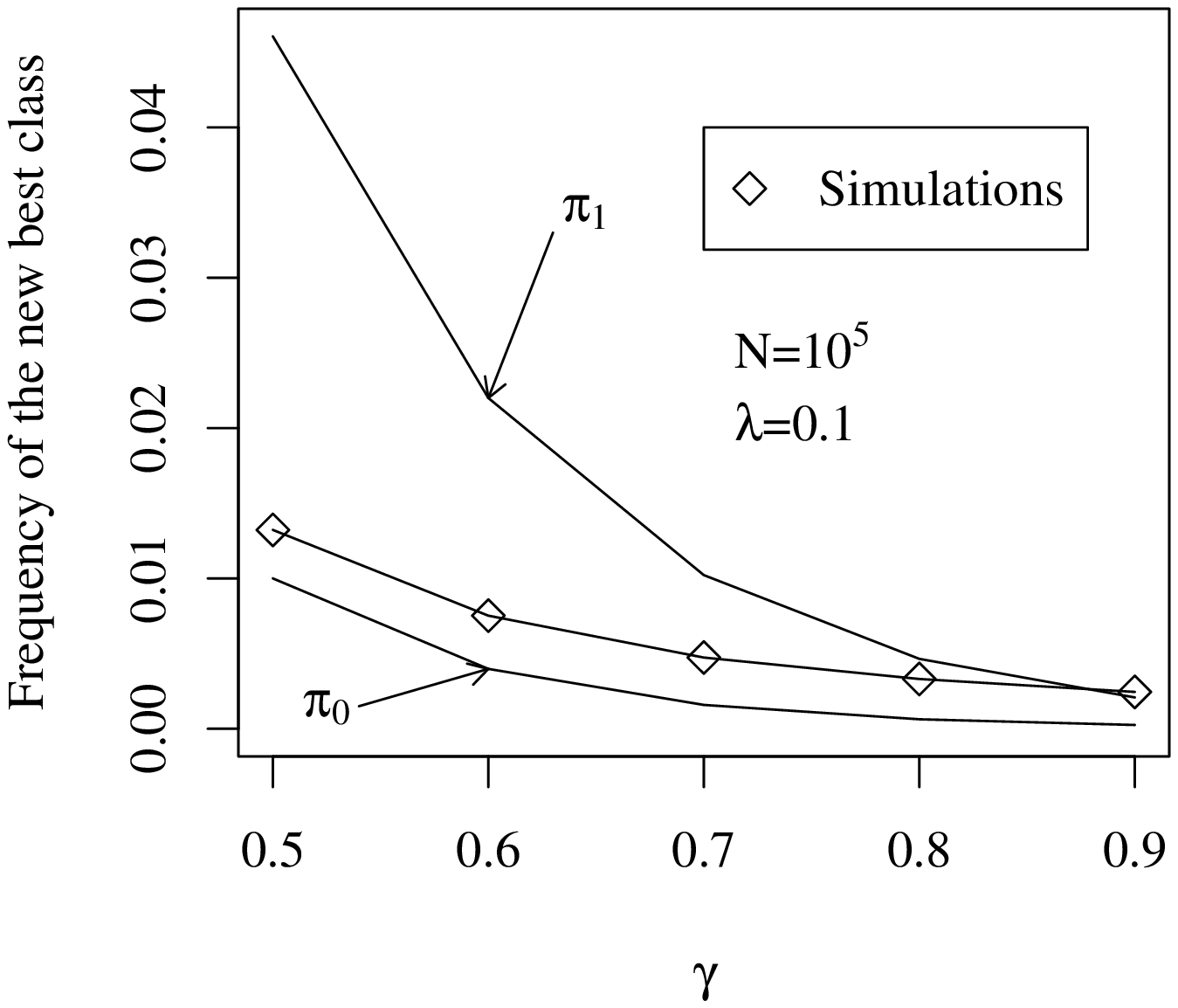}
  \end{center}
  \caption{\label{fig:bestClickTimes}Our heuristic that observed
    states come from a RPPA applies in particular at click times. (A)
    For small $\gamma$, i.e., rare clicking, the frequency of the best
    class is already close to $\pi_0$ while (B) it is close to $\pi_1$
    for larger $\gamma$, i.e., frequent clicking. Every plot is based
    on the simulation of $5\cdot 10^6$ generations.}
\end{figure}

Our prediction that we observe profiles which are well approximated by
a relaxed PPA applies especially well at click times. We check this
numerically by observing the frequency of the (new) best class at
click times; see Figure~\ref{fig:bestClickTimes}. For small $\gamma$
the ratchet clicks rarely and the system has some time to relax to its
new equilibrium even before the click of the ratchet. As a
consequence, we see that the frequency of the (new) best class at the
time of the click is already close to $\pi_0$. However, if $\gamma$ is
large and the ratchet clicks frequently, the dynamical system has no
time before the click to relax the system to the new equilibrium.
Therefore, we observe that the frequency of the new best class is
close to $\pi_1$.

\section{Discussion}
\label{discussion}
Haigh~(1978) was the first to attempt a rigorous mathematical analysis of 
the ideas of Muller~(1964).  However, in spite of the apparent simplicity of 
Haigh's mathematical formulation of the model, 
the exact rate of Muller's ratchet remains elusive.  
In this note, we have developed arguments in the
spirit of Haigh~(1978), Stephan et al.~(1993) and Gordo \&
Charlesworth~(2000) to give approximations for this rate.


Haigh gave
the empirical formula
$$ 4N\pi_0 + 7 \log \theta + \tfrac 2s -20$$
for the average time between clicks of the ratchet (where time is measured in
generations). 
A quantitative understanding of the rate was
first obtained by Stephan et al.~(1993) using diffusion approximations
and later extended by Gordo \& Charlesworth (2000). Both obtain the
diffusion \eqref{eq:three2b} as the main equation giving a valid
approximation for the frequency path of the best class.

The reasoning leading to \eqref{eq:three2b} in these papers is
twofold. Stephan et al.~(1993) and Stephan \& Kim~(2002) argue that
although fitness decreases by $se^{-\lambda}$ during one `cycle' of
the discrete ratchet model from \S\ref{haigh} (in which the system
advances from one Poisson equilibrium to the next), at the actual
click time only a fraction of the fitness has been lost.  They suggest
$kse^{-\lambda}$ for $k=0.5$ or $k=0.6$ as the loss of fitness at
click times. In other words they predict the functional relationship
$M_1(Y_0)$ discussed in \S\ref{one-dimensional diffusion} by linear
interpolation between $M_1(\pi_0)=\theta$ and $M_1(0) = \theta + k
e^{-\lambda}\approx \theta + k$; compare with Figure \ref{fig:three}.
On the other hand, Gordo \& Charlesworth~(2000) use a calculation of
Haigh which tells us that if the dynamical system~\eqref{DR} is
started in $\tilde\pi$ from \eqref{pitilde}, then at the end of phase
one (corresponding in the continuous setting, as we observed in
Remark~\ref{phase1}, to time $\tfrac 1s \log\theta$) we have
$sM_1\approx 1 - e^{-\lambda} (1+0.42 s)$.  This leads to the
approximation $M_1(1.6\pi_0) = \theta- 0.42$ and again interpolating
linearly using $M_1(\pi_0)=\theta$ gives \eqref{eq:three2b}.

Simulations show that \eqref{eq:three2b} provides a good approximation
to the rate of the ratchet for a wide range of parameters; see e.g.
Stephan and Kim~(2002).  The novelty in our work is that we derive
\eqref{eq:three2b} explicitly from the dynamical system. In
particular, we do not use a linear approximation, but instead derive a
functional linear relationship in Proposition \ref{p:relax}. In
addition, we clarify the r\^ole of the two different phases suggested
by Haigh. As simulations show, since phase one is fast, it is already
complete at the time when phase two starts. Therefore, in practice, we
observe states that are relaxed PPAs.

The drawback of our analysis is that we cannot give good arguments for
the choice of $A=1$ in \eqref{eq:threeb} and \eqref{eq:three2b}.
However, note that the choice of $A=1$ is essential to obtain the
prefactor of 0.58 in \eqref{eq:threeb}. E.g., if $\theta=10$, the
choice of $A=0.5$ leads to a prefactor of 0.13 while $A=2$ leads to
the prefactor of 0.95, neither of which fits with simulated data; see
Figure \ref{fig:phaseplane}.

We obtain two more diffusion approximations, which are valid in the
cases of frequent and rare clicking, respectively. In practice, both
play little r\^ole in the prediction of the rate of the ratchet. For
fast clicking, \eqref{eq:three2b} shows the same power law behaviour
as \eqref{eq:three2a} and rare clicks are never observed in
simulations.

Of course from a biological perspective our mathematical model is very
naive.  In particular, it is unnatural to suppose that each new
mutation confers the same selective disadvantage and, indeed, not all
mutations will be deleterious.  Moreover, if one is to argue that
Muller's ratchet explains the evolution of sex, then one has to
quantify the effect of recombination.  Such questions provide a rich,
but challenging, mathematical playground.

~

\noindent{\bf Acknowledgements}
We have discussed Muller's ratchet with many different people.  We are
especially indebted to Ellen Baake, Nick Barton, Matthias Birkner,
Charles Cuthbertson, Don Dawson, Wolfgang Stephan, Jay Taylor and Feng
Yu.

\end{document}